\documentclass[11pt]{article}
\usepackage{epsfig,url}
\usepackage{amssymb,amsmath}

\newcommand{\pr}{{\mathbb{P}}}
\newcommand{\eps}{\epsilon}
\newcommand{\lstar}{\ell^*}
\newcommand{\lstare}{\ell^*_e}

\small\normalsize

\newtheorem{lemma}{Lemma}[section]
\newtheorem{theorem}[lemma]{Theorem}

\newtheorem{conjecture}[lemma]{Conjecture}
\newtheorem{question}[lemma]{Question}
\newtheorem{corollary}[lemma]{Corollary}

\newtheorem{remark}[lemma]{Remark}

\newtheorem{problem}[lemma]{Problem}
\newtheorem{claim}[lemma]{Claim}

\newcommand{\sfrac}[2]{%
  \raisebox{2pt}{$#1$\hskip-1pt}\big/\raisebox{-3pt}{\hskip-1pt$#2$}}
\newcommand{\bfracb}[2]{\Bigl(\dfrac{#1}{#2}\Bigr)}
\newcommand{\cpa}[1]{#1^{(1)}}
\newcommand{\cpb}[1]{#1^{(2)}}
\newcommand{\cpi}[1]{#1^{(i)}}
\newcommand{\half}{\tfrac12}
\newcommand{\MF}{\mathcal{F}}
\newcommand{\MP}{\mathcal{MP}}
\newcommand{\ov}[1]{\widehat{#1}}

\newcommand{\proofend}{\qquad\hspace*{\fill}\mbox{$\Box$}}
\newcommand{\rme}{\mathrm{e}}
\newcommand{\rup}[1]{\lceil{#1}\rceil}
\newcommand{\sm}{\!\setminus\!}

\newenvironment{proof}{\begin{trivlist}\item[]\textbf{Proof}\mbox{ \ }}%
  {\qquad\hspace*{\fill}$\Box$\end{trivlist}}
\newcommand{\qitee}[1]{\noindent\leavevmode\hangindent1.5\parindent%
  \noindent\hbox to1.5\parindent{#1\hss}\ignorespaces}

\DeclareMathOperator{\ch}{ch}
\DeclareMathOperator{\dist}{dist}

\title{\textbf{List Colouring Squares of Planar Graphs}}
\author{\quad\\
  Fr\'ed\'eric Havet\,$^1$, \hspace{3mm}Jan van den Heuvel\,$^2$,
  \hspace{3mm}Colin McDiarmid\,$^3$,
  \hspace{3mm}and\hspace{3mm}Bruce Reed\,$^{1,4}$\\[3mm]
  $^1$ \emph{Uniersit\'e C\^ote d'Azur, CNRS, I3S et INRIA, Projet COATI,
    Sophia-Antipolis, France}\\[1mm]
  $^2$ \emph{Department of Mathematics, London School of Economics and
    Political Science, U.K.}\\[1mm]
  $^3$ \emph{Department of Statistics, University of Oxford, Oxford,
    U.K.}\\[1mm]
  $^4$
  \emph{Department of Computer Science, McGill University, Montreal,
    Canada}}

\begin{document}

\maketitle
{\renewcommand{\thefootnote}{\relax} \footnotetext{Part of the research for
    this paper was done during a visit of FH, JvdH and BR to the Department
    of Applied Mathematics (KAM) at the Charles University of Prague;
    during a visit of FH, CMcD and BR to Pacific Institute of Mathematics
    Station in Banff (while attending a focused research group program);
    during a visit of JvdH to INRIA in Sophia-Antipolis (funded by Hubert
    Curien programme Alliance 15130TD and the British Council Alliance
    programme), and during a visit of CMcD and BR to IMPA in Rio de
    Janiero. The authors would like to thank all institutes involved for
    their support and hospitality.

    \hskip7.5pt This research is also partially supported by the ANR under
    contract STINT ANR-13-BS02-0007.

    \hskip7.5pt Email: \texttt{frederic.havet@cnrs.fr},
    \texttt{j.van-den-heuvel@lse.ac.uk},
    \texttt{cmcd@stats.ox.ac.uk},\\
    \texttt{breed@cs.mcgill.ca}.}}

\begin{abstract}
  \noindent
  In 1977, Wegner conjectured that the chromatic number of the square of
  every planar graph with maximum degree $\Delta\ge8$ is at most
  $\bigl\lfloor\frac32\Delta\bigr\rfloor+1$. We show that it is at most
  $\frac32\Delta(1+o(1))$ (where the $o(1)$ is as $\Delta\to+\infty$), and
  indeed that this is true for the list chromatic number and for more
  general classes of graphs.
\end{abstract}


\section{Introduction}\label{intr}

Most of the terminology and notation we use in this paper is standard and
can be found in any text book on graph theory (such as~\cite{BoMu76}
or~\cite{Die05}). All our graphs and multigraphs will be finite. A
\emph{multigraph} can have multiple edges; a \emph{graph} is supposed to be
simple. We will not allow loops.

The \emph{degree} of a vertex is the number of edges incident with that
vertex. We require all colourings, whether we are discussing vertex, edge
or list colouring, to be \emph{proper}: neighbouring objects must receive
different colours. We also always assume that colours are integers, which
allows us to talk about the ``distance'' $|\gamma_1-\gamma_2|$ between two
colours~$\gamma_1,\gamma_2$.

Given a graph~$G$, the \emph{chromatic number} of~$G$, denoted~$\chi(G)$,
is the minimum number of colours required so that we can properly colour
its vertices using those colours. If we colour the edges of~$G$, we get the
\emph{chromatic index}, denoted~$\chi'(G)$.

Given a list~$L(v)$ of colours for each vertex $v$ of $G$, we say a
colouring is \emph{acceptable} (with respect to the lists) if it is proper
and every vertex gets assigned a colour from its own private list. The
\emph{list chromatic number} or \emph{choice number}~$\ch(G)$ is the
minimum value~$k$ such that, if we give each vertex of~$G$ a list of size
$k$, then there is an acceptable colouring. The \textit{list chromatic
  index} is defined analogously for edges. See~\cite{Wo01} for a survey of
research on list colouring of graphs. Note that the list $L(v)$ is really
just a set, but as is standard we refer to it as a list.


\subsection{Colouring the Square of a Graph}

Given a graph~$G$, the \emph{square of~$G$}, denoted~$G^2$, is the graph
with the same vertex set as~$G$ and with an edge between each pair of
distinct vertices that have distance at most two in~$G$. If~$G$ has maximum
degree~$\Delta$, then a vertex colouring of its square will need at least
$\Delta+1$ colours; the greedy algorithm shows it is always possible with
$\Delta^2+1$ colours. Diameter two cages such as the 5-cycle, the Petersen
graph and the Hoffman-Singleton graph (see \cite[page~84]{BoMu76}) show
that there exist graphs that in fact require $\Delta^2+1$ colours, for
$\Delta=2,3,7$, and possibly one for $\Delta=57$.

We are particularly interested in planar graphs. The celebrated Four Colour
Theorem by Appel and Haken~\cite{ApHa77a,ApHa77b,ApHa89} states that
$\chi(G)\le4$ for planar graphs $G$. Regarding the chromatic number of the
square of a planar graph, Wegner~\cite{Weg77} posed the following
conjecture (see also the book of Jensen and Toft~\cite[Section
2.18]{JeTo95}), suggesting that for planar graphs far less than
$\Delta^2+1$ colours suffice.

\begin{conjecture}[Wegner~\cite{Weg77}]\label{conj.W1}\mbox{}\\*
  For a planar graph~$G$ with maximum degree~$\Delta$,
  \[\chi(G^2)\:\le\:\left\{
    \begin{array}{ll}
      7,&\text{if $\Delta=3$,}\\
      \Delta+5,&\text{if $4\le\Delta\le7$,}\\
      \bigl\lfloor\frac32\Delta\bigr\rfloor+1,&\text{if $\Delta\ge8$.}
    \end{array}\right.\]
\end{conjecture}

\noindent
Wegner also gave examples showing that these bounds would be tight. For
even $\Delta\ge8$, these examples are sketched in Figure~\ref{tightfig}.
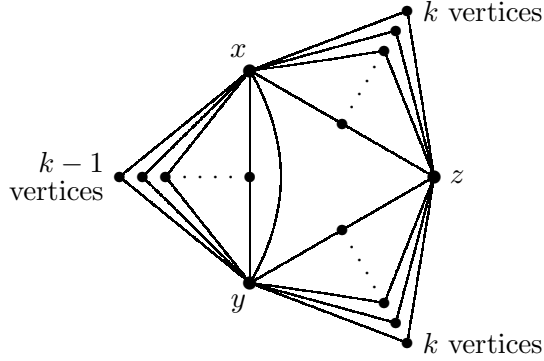
\begin{figure}[!hbtp]
  \begin{center}
    \unitlength0.51mm
    \begin{picture}(82,94)(-50,-45)
      \put(-50,0){\circle*{3}}\put(-44,0){\circle*{3}}
      \put(-38,0){\circle*{3}}\put(-33,0){\circle*{1}}
      \put(-29,0){\circle*{1}}\put(-25,0){\circle*{1}}
      \put(-21,0){\circle*{1}}\put(-16,0){\circle*{3}}
      \put(25,43.30){\circle*{3}}\put(22,38.11){\circle*{3}}
      \put(19,32.91){\circle*{3}}\put(16.5,28.58){\circle*{1}}
      \put(14.5,25.11){\circle*{1}}\put(12.5,21.65){\circle*{1}}
      \put(10.5,18.19){\circle*{1}}\put(8,13.86){\circle*{3}}
      \put(25,-43.30){\circle*{3}}\put(22,-38.11){\circle*{3}}
      \put(19,-32.91){\circle*{3}}\put(16.5,-28.58){\circle*{1}}
      \put(14.5,-25.11){\circle*{1}}\put(12.5,-21.65){\circle*{1}}
      \put(10.5,-18.19){\circle*{1}}\put(8,-13.86){\circle*{3}}
      \put(-16,27.71){\circle*{3.5}}\put(-16,-27.71){\circle*{3.5}}
      \put(32,0){\circle*{3.5}}
      \qbezier(-16,27.71)(-16,13.855)(-16,0)
      \qbezier(-16,27.71)(-27,13.855)(-38,0)
      \qbezier(-16,27.71)(-30,13.855)(-44,0)
      \qbezier(-16,27.71)(-33,13.855)(-50,0)
      \qbezier(-16,-27.71)(-16,-13.855)(-16,0)
      \qbezier(-16,-27.71)(-27,-13.855)(-38,0)
      \qbezier(-16,-27.71)(-30,-13.855)(-44,0)
      \qbezier(-16,-27.71)(-33,-13.855)(-50,0)
      \qbezier(-16,27.71)(-4,20.785)(8,13.86)
      \qbezier(-16,27.71)(1.5,29.95)(19,32.91)
      \qbezier(-16,27.71)(3,32.91)(22,38.11)
      \qbezier(-16,27.71)(4.5,35.505)(25,43.30)
      \qbezier(-16,-27.71)(-4,-20.785)(8,-13.86)
      \qbezier(-16,-27.71)(1.5,-29.95)(19,-32.91)
      \qbezier(-16,-27.71)(3,-32.91)(22,-38.11)
      \qbezier(-16,-27.71)(4.5,-35.505)(25,-43.30)
      \qbezier(32,0)(20,6.93)(8,13.86)
      \qbezier(32,0)(25.5,16.455)(19,32.91)
      \qbezier(32,0)(27,19.055)(22,38.11)
      \qbezier(32,0)(28.5,21.65)(25,43.30)
      \qbezier(32,0)(20,-6.93)(8,-13.86)
      \qbezier(32,0)(25.5,-16.455)(19,-32.91)
      \qbezier(32,0)(27,-19.055)(22,-38.11)
      \qbezier(32,0)(28.5,-21.65)(25,-43.30)
      \qbezier(-16,27.71)(0,0)(-16,-27.71)
      \put(-54,3.5){\makebox(0,0)[r]{$k-1$}}
      \put(-54,-3.5){\makebox(0,0)[r]{vertices}}
      \put(29,43.40){\makebox(0,0)[l]{$k$ vertices}}
      \put(29,-43.40){\makebox(0,0)[l]{$k$ vertices}}
      \put(38,0){\makebox(0,0){$z$}}
      \put(-19,32.91){\makebox(0,0){$x$}}
      \put(-19,-32.91){\makebox(0,0){$y$}}
    \end{picture}
    \caption{The planar graph $G_k$.}\label{tightfig}
  \end{center}
\end{figure}
The graph~$G_k$ consists of three vertices $x$, $y$ and~$z$ together with
$3k-1$ additional vertices with degree two, such that~$z$ has~$k$ common
neighbours with~$x$ and~$k$ common neighbours with~$y$, and $x$ and~$y$ are
adjacent and have $k-1$ common neighbours. This graph has maximum
degree~$2k$ and yet all the vertices except~$z$ are adjacent in its square.
Hence to colour these $3k+1$ vertices, we need at least
$3k+1=\frac32\Delta+1$ colours.

Kostochka and Woodall~\cite{KoWo01} conjectured that for every square of a
graph the list chromatic number equals the chromatic number. This
conjecture was first disproved by Kim and Park~\cite{KP15}. Since then more
counterexamples have been found \cite{KKP15,KP15b,KSRY14}. All these
counterexamples are not planar, which gives us hope that Kostochka and
Woodall's original conjecture is true for planar graphs.

\begin{conjecture}\label{conj.W2}\mbox{}\\*
  For a planar graph~$G$ with maximum degree~$\Delta$,
  \[\ch(G^2)\:\le\:\left\{
    \begin{array}{ll}
      7,&\text{if $\Delta=3$,}\\
      \Delta+5,&\text{if $4\le\Delta\le7$,}\\
      \bigl\lfloor\frac32\Delta\bigr\rfloor+1,&\text{if $\Delta\ge8$.}
    \end{array}\right.\]
\end{conjecture}

\noindent
Wegner also showed that if~$G$ is a planar graph with $\Delta=3$,
then~$G^2$ can be 8-coloured. Thomassen~\cite{Tho06} established Wegner's
conjecture for $\Delta=3$ using an involved structural result on subcubic
(i.e.\ with $\Delta\le3$) graphs; while Hartke \emph{et al}.~\cite{HJT16}
proved the same using the discharging method and a serious amount of
computer time. Cranston and Kim~\cite{CrKi} showed that the square of every
connected graph (not necessarily planar) which is subcubic is 8-choosable,
except for the Petersen graph. However, the 7-choosability of the squares
of subcubic planar graphs is still open.

The first upper bound on~$\chi(G^2)$ for planar graphs that is linear
in~$\Delta$, namely $\chi(G^2)\le 8\Delta-22$, was implicit in the work of
Jonas~\cite{Jon93}. (The results in~\cite{Jon93} deal with
$L(2,1)$-labellings, see below, but the proofs are easily seen to be
applicable to colouring the square of a graph as well.) This bound was
later improved by Wong~\cite{Won96} to $\chi(G^2)\le 3\Delta+5$, and then
by Van den Heuvel and McGuinness~\cite{vdHMG} to $\chi(G^2)\le 2\Delta+25$.
Better bounds were then obtained for large values of~$\Delta$. It was shown
that $\chi(G^2)\le\bigl\lceil\frac95\Delta\bigr\rceil+1$ for $\Delta\ge750$
by Agnarsson and Halld\'orsson~\cite{AgHa00}, and the same bound for
$\Delta\ge47$ by Borodin \emph{et al}.~\cite{Bor+}. Finally, the
asymptotically best known upper bound so far has been obtained by Molloy
and Salavatipour~\cite{MoSa02} as a special case of Theorem~\ref{MoSa}
below.

\begin{theorem}[Molloy and Salavatipour~\cite{MoSa02}]\label{MoSa2}
  \mbox{}\\*
  For a planar graph~$G$ with maximum degree $\Delta$,
  \[\chi(G^2)\:\le\: \tfrac53\Delta +78.\]
\end{theorem}

\noindent
As mentioned in \cite{MoSa02}, the constant 78 can be reduced for
sufficiently large~$\Delta$; the paper improves it to~24 when
$\Delta\ge241$.

In this paper we prove the following theorem.

\begin{theorem}\label{mtplanar}\mbox{}\\*
  The square of every planar graph~$G$ with maximum degree~$\Delta$ has
  list chromatic number at most $(1+o(1))\frac32\Delta$. Moreover, given
  lists of this size, there is an acceptable colouring in which the colours
  on every pair of adjacent vertices of~$G$ differ by at
  least~$\Delta^{1/4}$.
\end{theorem}

\noindent
A more precise statement is as follows. For each $\eps>0$, there is
a~$\Delta_\eps$ such that for every $\Delta\ge\Delta_\eps$ we have: for
every planar graph~$G$ with maximum degree at most~$\Delta$, and for all
vertex lists each of size at least $\bigl(\frac32+\eps\bigr)\Delta$, there
is an acceptable colouring of~$G$, with the further property that the
colours on every pair of adjacent vertices of~$G$ differ by at
least~$\Delta^{1/4}$.

The~$o(1)$ term in the theorem is as $\Delta\longrightarrow+\infty$. The
first order term $\frac32\Delta$ in Theorem~\ref{mt} is best possible, as
the examples in Figure~\ref{tightfig} show. On the other hand, the
term~$\Delta^{1/4}$ is probably far from best possible; it was chosen to
keep the proof simple. The main point, to our minds, is that this parameter
tends to infinity as $\Delta\longrightarrow+\infty$.

In~\cite{AEH2013}, the first part of Theorem~\ref{mtplanar} is extended to
graphs $G$ embeddable in any fixed surface\footnote{We note that
  \cite{AEH2013} was written after the results in this paper were obtained,
  due to the lengthy amount of time this paper has spent in the revision
  process (which is the fault of the authors), and combines the techniques
  developed in this paper with other arguments}. That paper also considers
the more general framework of $\Sigma$-colourings, where for each
vertex~$v$ a subset $\Sigma(v)\subseteq N_G(v)$ of the neighbourhood of~$v$
is given, and two vertices $u,w$ only need to receive a different colour if
$uw\in E(G)$ or $u,w\in\Sigma(v)$ for some~$v$. This concept unifies
ordinary colourings (taking $\Sigma(v)=\varnothing$ for all~$v$) and
colourings of the square~$G^2$ (taking $\Sigma(v)=N_G(v)$ for all~$v$). It
also includes so-called \emph{cyclic colourings} of graphs that are
embedded in a surface, where vertices that share a face must be coloured
differently.

Here we extend Theorem~\ref{mtplanar} to every \emph{nice family of
  graphs}, which are those minor-closed families of graphs such that there
is some $k$ for which the complete bipartite graph $K_{3,k}$ is not in the
family.

\begin{theorem}\label{mt}\mbox{}\\*
  Let~$\MF$ be a nice family of graphs. The square of every graph~$G$
  in~$\MF$ with maximum degree~$\Delta$ has list chromatic number at most
  $\bigl(\frac32+o(1)\bigr)\Delta$. Moreover, given lists of this size,
  there is a proper colouring in which the colours on every pair of
  adjacent vertices of~$G$ differ by at least~$\Delta^{1/4}$.
\end{theorem}

\noindent
Kuratowski's theorem tells us that planar graphs form a nice family. So do
graphs which are embeddable in a fixed surface. For, by Euler's formula, if
a bipartite graph with $n$ vertices and~$e$ edges embeds in a surface
$\Sigma$ of Euler genus $g$, then $e\le 2(n+g-2)$; and so $K_{3,k}$ does
not embed in $\Sigma$ if $k>2g+2$.

Note that~$K_{3,3}$ has~$K_4$ as a minor, and so $K_4$-minor-free graphs
(that is, series-parallel graphs) form a nice class. Lih, Wang and
Zhu~\cite{LWZ03} showed that the square of a $K_4$-minor-free graph with
maximum degree~$\Delta$ has chromatic number at most
$\bigl\lfloor\frac32\Delta\bigr\rfloor+1$ if $\Delta\ge4$ and $\Delta+3$ if
$\Delta=2,3$. The same bounds, but then for the list chromatic number of
the square of a $K_4$-minor-free graph, were proved by Hetherington and
Woodall~\cite{HW07}.


\subsection{$L(p,q)$-Labellings of Graphs}

Vertex colourings of squares of graphs can be considered a special case of
a more general concept: $L(p,q)$-labellings of graphs. This topic takes
some of its inspiration from so-called channel assignment problems in radio
or cellular phone networks, see for example~\cite{LeeseHurley}. The basic
channel assignment problem is the following: we need to assign radio
frequency channels to transmitters (each gets one channel which corresponds
to an integer). In order to avoid interference, if two transmitters are
very close, then the separation of the channels assigned to them has to be
large enough. Moreover, if two transmitters are close but not very close,
then they must also receive channels that are sufficiently far apart.

An idealised version of such a problem may be modelled by
$L(p,q)$-labellings of a graph~$G$, where~$p$ and~$q$ are non-negative
integers. The vertices of this graph correspond to the transmitters and two
vertices are linked by an edge if they are very close. Two vertices are
then considered close if they are at distance two in the graph. Let
$\dist(u,v)$ denote the distance between the two vertices~$u$ and~$v$. An
\emph{$L(p,q)$-labelling of~$G$} is an integer assignment~$f$ to the vertex
set~$V(G)$ such that:

\smallskip
\qitee{\quad$\bullet$}$|f(u)-f(v)|\ge p\;$ if $\dist(u,v)=1$, and

\smallskip
\qitee{\quad$\bullet$}$|f(u)-f(v)|\ge q \;$ if $\dist(u,v)=2$.

\smallskip\noindent
It is natural to assume that $p\ge q$, and we do so throughout.

The \emph{span} of~$f$ is the difference between the largest and the
smallest labels of~$f$ plus one. The \emph{$\lambda_{p,q}$-number} of~$G$,
denoted by $\lambda_{p,q}(G)$, is the minimum span over all
$L(p,q)$-labellings of~$G$.

The problem of determining $\lambda_{p,q}(G)$ has been studied for some
specific classes of graphs (see the survey of Yeh~\cite{Yeh06}).
Generalisations of $L(p,q)$-labellings have also been studied in which a
minimum gap of~$p_i$ is required for channels assigned to vertices at
distance~$i$, for several values $i=1,2,\ldots$ (see for
example~\cite{Kral} or~\cite{LiZh05}).

Moreover, very often, because of technical reasons or dynamicity, the set
of channels available varies from transmitter to transmitter. Therefore one
has to consider the list version of $L(p,q)$-labellings. A \emph{$k$-list
  assignment~$L$} of a graph is a function which assigns to each vertex~$v$
of the graph a list~$L(v)$ of~$k$ prescribed integers. Given a graph~$G$,
the \emph{list $\lambda_{p,q}$-number}, denoted $\lambda^l_{p,q}(G)$, is
the smallest integer~$k$ such that, for every $k$-list assignment~$L$
of~$G$, there exists an $L(p,q)$-labelling~$f$ such that $f(v)\in L(v)$ for
every vertex~$v$. Surprisingly, list $L(p,q)$-labellings have received very
little attention and appear only quite recently in the
literature~\cite{KSTM05}. However, some of the proofs for
$L(p,q)$-labellings also work for list $L(p,q)$-labellings.

Note that $L(1,0)$-labellings of~$G$ correspond to ordinary vertex
colourings of~$G$ and $L(1,1)$-labellings of~$G$ to vertex colourings of
the square of~$G$: thus $\lambda_{1,0}(G)=\chi(G)$,
$\lambda^l_{1,0}(G)=\ch(G)$, $\lambda_{1,1}(G)=\chi(G^2)$, and
$\lambda^l_{1,1}(G)=\ch(G^2)$.

It is well known that for a graph~$G$ with clique number~$\omega$ (the size
of a maximum clique in~$G$) and maximum degree~$\Delta$ we have
$\omega\le\chi(G)\le\ch(G)\le\Delta+1$. Similar easy inequalities may be
obtained for $L(p,q)$-labellings:
\[q\,\omega(G^2)-q+1\:\le\: \lambda_{p,q}(G)\:\le\:
d\lambda^l_{p,q}(G)\:\le\: p\,\Delta(G^2)+1.\]
(Recall that we assume throughout that $p\ge q$.) As
$\omega(G^2)\ge\Delta(G)+1$, the previous inequality gives
$\lambda_{p,q}(G)\ge q\Delta+1$. However, a straightforward argument shows
that in fact we must have $\lambda_{p,q}(G)\ge q\Delta+p-q+1$. In the same
way, $\Delta(G^2)\le(\Delta(G))^2$ so $\lambda^l_{p,q}(G)\le
p(\Delta(G))^2+1$. The ``many-passes'' greedy algorithm (see~\cite{McD03})
gives the alternative bound
\[\lambda^l_{p,q}(G)\:\le\: q\Delta(G)(\Delta(G)-1)+p\Delta(G)+1\:=\:
q(\Delta(G))^2+(p-q)\Delta(G)+1.\]

Because for many large-scale networks the transmitters are laid out on the
surface of the earth, $L(p,q)$-labellings of planar graphs are of
particular interest. There are planar graphs for which
$\lambda_{p,q}\ge\frac32q\Delta+c(p,q)$, where $c(p,q)$ is a constant
depending on~$p$ and~$q$. We already saw some of those examples in
Figure~\ref{tightfig}. The graph~$G_k$ has maximum degree~$2k$ and yet its
square contains a clique with $3k+1$ vertices (all the vertices
except~$z$). Labelling the vertices in the clique already requires a span
of at least $q\cdot 3k+1=\frac32q\Delta+1$.

A first upper bound on $\lambda_{p,q}(G)$, for planar graphs~$G$ and
positive integers $p\ge q$ was proved by Van den Heuvel and
McGuinness~\cite{vdHMG}: $\lambda_{p,q}(G)\le 2(2q-1)\Delta+10p+38q-24$.
Molloy and Salavatipour~\cite{MoSa02} improved this bound by showing the
following.

\begin{theorem}[Molloy and Salavatipour~\cite{MoSa02}]\label{MoSa}
  \mbox{}\\*
  For positive integers $p\ge q$, and a planar graph~$G$ with maximum
  degree $\Delta$,
  \[\lambda_{p,q}(G)\:\le\:
  q\bigl\lceil\tfrac53\Delta\bigr\rceil+18p+77q-18.\]
\end{theorem}

\noindent
Moreover, they described an $O(n^2)$ time algorithm for finding an
$L(p,q)$-labelling with span at most the bound in their theorem.

As a corollary to our main result Theorem~\ref{mt} we get that, for any
fixed~$p$ and every nice family~$\MF$ of graphs, we have
$\lambda^l_{p,1}(G)\le(1+o(1))\frac32\Delta(G)$ for $G\in\MF$. Taking an
$L(\rup{p/k},\rup{q/k})$-labelling and multiplying each label by~$k$, for
some positive integer~$k$, we obtain an $L(p,q)$-labelling. This gives the
following corollary.

\begin{corollary}\mbox{}\\*
  Let~$\MF$ be a nice family of graphs and let $p\ge q$ be positive
  integers. Then for graphs~$G$ in~$\MF$ we have
  $\lambda_{p,q}(G)\le(1+o(1))\frac32q\Delta(G)$.
\end{corollary}

\noindent
Note that the examples discussed earlier show that for each positive
integer~$q$ the factor $\frac32q$ is optimal.


\section{Nice Families of Graphs}\label{nice}

Recall that we call a family~$\MF$ of graphs \emph{nice} if (a) it is
closed under taking minors and (b) there is some $k$ for which $K_{3,k}$ is
not in the family. In this section we prove a number of properties of nice
families, eventually showing that we obtain an equivalent definition if we
replace condition~(b) by the following condition:

\smallskip\noindent
(c) \ there is a constant~$\beta_\MF$ such that for any graph $G\in\MF$ and
any vertex set $B\subseteq V(G)$, if we let $A$ be the set of vertices in
$V(G)\sm B$ which have at least three neighbours in~$B$, then the number of
edges between $A$ and $B$ is at most $\beta_\MF|B|$.

\smallskip
To prove this equivalence, we need the following result.

\begin{theorem}[Mader~\cite{Mad68}]\label{adbounded}\mbox{}\\*
  For any graph~$H$, there is a constant~$C_H$ such that every $H$-minor
  free graph has average degree at most~$C_H$.
\end{theorem}

\noindent
In the proof of Theorem~\ref{adbounded}, Mader showed that $C_H\le
c|V(H)|\log|V(H)|$, for some constant~$c$. This upper bound was later
lowered independently by Kostochka~\cite{Kos84} and Thomason~\cite{Tho84}
to $C_H\le c'|V(H)|\sqrt{\log|V(H)|}$, for some constant~$c'$.

\begin{corollary}\label{fewtriangles}\mbox{}\\*
  Any~$H$-minor-free graph with~$n$ vertices has at most
  $\displaystyle\binom{\lfloor C_H\rfloor}{2}\cdot n$ triangles.
\end{corollary}

\begin{proof}
  We prove the result by induction on~$n$, the result holding trivially if
  $n\le2$. Let~$G$ be an $H$-minor-free graph with~$n$ vertices. By
  Theorem~\ref{adbounded}, its average degree is at most~$C_H$. So~$G$ has
  a vertex~$v$ with degree at most~$\lfloor C_H\rfloor$. The vertex $v$ is
  in at most $\binom{\lfloor C_H\rfloor}{2}$ triangles. By induction, $G-v$
  has at most $\binom{\lfloor C_H\rfloor}{2}\cdot(n-1)$ triangles.
  Hence~$G$ has at most $\binom{\lfloor C_H\rfloor}{2}\cdot n$ triangles.
\end{proof}

\noindent
For an extension of this result see Lemma 2.1 of Norine \emph{et
  al}.~\cite{NSTW}.

\begin{theorem}\label{K3nfree}\mbox{}\\*
  A class $\MF$ of graphs is nice if and only if it is minor-closed and
  satisfies condition (c).
\end{theorem}

\begin{proof}
  First suppose that (c) holds for $\MF$. By taking~$B$ the set of three
  vertices in~$K_{3,k}$ from one part of the bipartition, and~$A$ the
  remaining~$k$ vertices, we see that $K_{3,k}$ cannot be in $\MF$ for
  $k>\beta_\MF$. It follows that every graph in~$\MF$ is
  $K_{3,k}$-minor-free if $k>\beta_\MF$.

  Next suppose that $\MF$ is a minor-closed family not containing $K_{3,k}$
  for some $k$. We want to prove that (c) holds for~$\MF$. Note that by
  Theorem~\ref{adbounded}, the average degree of a $K_{3,k}$-minor-free
  graph is bounded by some constant~$C_k$.

  Let $G\in\MF$, let $B$ be a set of vertices of~$G$, and let~$A$ be the
  set of vertices in $V\sm B$ having at least three neighbours in~$B$.
  Construct a graph~$H$ with vertex set~$B$ as follows: For each vertex
  of~$A$, one after another, if two of its neighbours in~$B$ are not yet
  adjacent in~$H$, choose a pair of those non-adjacent neighbours and add
  an edge between them.

  Let $A'\subseteq A$ be the set of vertices for which an edge has been
  added to~$H$, and set $A''=A\sm A'$. Then~$H$ is $K_{3,k}$-minor-free
  because~$G$ was, and hence $|A'|=|E(H)|\le\half C_k|B|$. Now for every
  vertex $a\in A''$, the neighbours of~$a$ in~$B$ form a clique in~$H$
  (otherwise we would have used~$a$ to link two of its neighbours in~$B$).
  Moreover, $k$ vertices of~$A''$ may not be complete to (that is, adjacent
  to each vertex of) the same triangle of~$H$, since otherwise~$G$ would
  contain a $K_{3,k}$-minor. Hence~$|A''|$ is at most $k-1$ times the
  number of triangles in~$H$, which is at most
  $\binom{\lfloor C_H\rfloor}{2}|B|$ by Corollary~\ref{fewtriangles}. We
  find that $|A''|\le(k-1)\binom{\lfloor C_H\rfloor}{2}|B|$, and hence
  $|A|=|A'|+|A''|\le \bigl(\half C_k+(k-1)\binom{\lfloor
    C_H\rfloor}{2}\bigr)|B|$.

  Since the subgraph of~$G$ induced on $A\cup B$ is $K_{3,k}$-minor-free,
  there are at most $\half C_k(|A|+|B|)$ edges between~$A$ and~$B$; that
  is, at most
  $\half C_k \bigl(\half C_k+
  (k-1)\binom{\lfloor C_H\rfloor}{2}+1\bigr)|B|$. So we are done with
  $\beta_\MF=
  \half C_k\bigl(\half C_k+(k-1)\binom{\lfloor C_H\rfloor}{2}+1\bigr)$.
\end{proof}


\section{Overview of the proof of Theorem~\ref{mt}}\label{sketch}

To prove Theorem~\ref{mt}, for a fixed nice family $\MF$, we need to show
that for every $\eps>0$ there is a~$\Delta_\eps$ such that for every
$\Delta\ge\Delta_\eps$ we have: for every graph $G\in\MF$ with maximum
degree at most~$\Delta$, given lists of size
\[\lstar\:=\: \lstar(\Delta,\eps)\:=\:
\bigl\lfloor\bigl(\tfrac32+\eps\bigr)\Delta\bigr\rfloor\]
for each vertex~$v$ of~$G$, we can find the desired colouring.

Given a graph $G$ with vertex set $V$, and $R\subseteq V$, we write
$G\!-\!R$ for the graph obtained from~$G$ by deleting the vertices in $R$
(and any incident edges), and write $G\!-\!v$ for $G\!-\!\{v\}$. Similarly,
we may write $V\sm v$ for $V\sm\{v\}$.

We proceed by induction on the number of vertices of~$G$. Our proof is a
recursive algorithm. In each iteration, we split off a set~$R$ of vertices
of the graph which are easy to handle, recursively colour $G^2\!-\!R$
(which we can do by the induction hypothesis), and then extend this
colouring to the vertices of~$R$. In extending the colouring, we must
ensure that no vertex~$v$ in~$R$ receives a colour which is either used on
a vertex in $V\sm R$ which is adjacent to~$v$ in~$G^2$ or is too close to a
colour on a vertex in $V\sm R$ which is adjacent to~$v$ in~$G$. Thus, we
modify the list~$L(v)$ of colours available for~$v$ by deleting those which
are forbidden because of such neighbours.

We note that $(G\!-\!R)^2$ need not be equal to $G^2\!-\!R$, as there may
be non-adjacent vertices of $G\!-\!R$ with a common neighbour in~$R$ but no
common neighbour in $G\!-\!R$. When choosing~$R$, we need to ensure that we
can construct a graph~$G_1$ in~$\MF$ on $V\sm R$ such that $G^2\!-\!R$ is a
subgraph of $G_1^2$. We also need to ensure that the connections
between~$R$ and $V\sm R$ are limited, so that the modified lists used when
list colouring the induced subgraph~$G^2[R]$ are still reasonably large.
Finally, we will want~$G^2[R]$ to have a simple structure so that we can
prove that we can list colour it as desired.

We begin with a simple example of such a set~$R$. We say a vertex~$v$
of~$G$ is \emph{removable} if it has at most~$\Delta^{1/4}$ neighbours
in~$G$ and at most two neighbours in~$G$ which have degree at
least~$\Delta^{1/4}$. We note that if~$v$ is a (removable) vertex with at
most one neighbour, then $(G\!-\!v)^2$ is $G^2[V\sm v]$, while if $v$ has
exactly two neighbours~$x$ and~$y$, then forming $G_1$ from $G\!-\!v$ by
adding an edge between~$x$ and~$y$ if they are not already adjacent, we
have that~$G_1$ is in~$\MF$ and $G^2\!-\!v$ is a subgraph of $G_1^2$. On
the other hand, if~$v$ is a removable vertex with at least three
neighbours, then it must have a neighbour~$w$ with degree at
most~$\Delta^{1/4}$. In this case, the graph~$G_2$ obtained from $G\!-\!v$
by adding an edge from~$w$ to every other neighbour of~$v$ in~$G$ is a
graph with maximum degree at most~$\Delta$ such that $G^2\!-\!v$ is a
subgraph of $G_2^2$. Furthermore, $G_2\in\MF$ as it is obtained from~$G$ by
contracting the edge~$wv$.

Thus, for any removable vertex~$v$, we can recursively list colour
$G^2\!-\!v$ using our algorithm. If, in addition, $v$ has at most
$\lstar-1-2\Delta^{1/2}$ neighbours in~$G^2$, then there will be a colour
in~$L(v)$ which appears on no vertex adjacent to~$v$ in~$G^2$ and is not
within~$\Delta^{1/4}$ of any colour assigned to a neighbour of~$v$ in~$G$.
To complete the colouring we give~$v$ any such colour.

The above remarks show that no minimal counterexample to our theorem can
contain a removable vertex of low degree in $G^2$. We are about to describe
another, more complicated, reduction we will use. It relies on the
following easy result.

\begin{lemma}\label{lemrem}\mbox{}\\*
  If~$R$ is a set of removable vertices of~$G$, then there is a graph
  $G_1\in\MF$ with vertex set $V\sm R$ and maximum degree at most~$\Delta$
  such that $G^2\!-\!R$ is a subgraph of $G_1^2$.
\end{lemma}

\begin{proof}
  For each $v\in R$ with at least three neighbours in $V\sm R$, choose one
  of these neighbours with degree less than~$\Delta^{1/4}$ onto which we
  will contract~$v$. Add an edge between the two neighbours of any vertex
  in~$R$ with exactly two neighbours in $V\sm R$ (if they are not already
  adjacent). The degree of a vertex~$x$ in the resultant graph~$G_1$ is at
  most $\max\{\Delta^{1/2},d_G(x)\}$.
\end{proof}

\noindent
For any multigraph~$H$, we let~$H^*$ be the graph obtained from~$H$ by
subdividing each edge exactly once. For each edge~$e$ of~$H$, we let~$e^*$
be the vertex of~$H^*$ which we placed in the middle of~$e$ and we
let~$E^*$ be the set of all such vertices. We call this set of vertices
corresponding to the edges of~$H$ the \emph{core of~$H^*$}.

A \emph{removable copy of~$H^*$} is a subgraph of~$G$ isomorphic to~$H^*$
such that the vertices of~$G$ corresponding to the vertices of the core
of~$H^*$ are removable, and each vertex of~$H^*$ corresponding to a vertex
of~$H$ (i.e.\ not in the core) has degree at least~$\Delta^{1/4}$ (in $G$).
It follows that if $v^*,w^*,e^*$ are vertices of $G$ corresponding to the
vertices $v,w$ and edge $e=vw$ of $H$, then $e^*$ and all its neighbours
other than $v^*$ and $w^*$ have degree at most $\Delta^{1/4}$.

Note that the subgraph~$J$ of~$G^2$ induced by the core of some copy
of~$H^*$ in~$G$ contains a subgraph isomorphic to~$L(H)$, the line graph
of~$H$. So the list chromatic number of~$J$ is at least the list chromatic
number of~$L(H)$. If the copy of $H^*$ is removable, then removing the
edges of this copy of~$L(H)$ from~$J$ yields a graph $J'$ which has as
vertex set the core of $H^*$, and has maximum degree at
most~$\Delta^{1/2}$. (To see this, note that if $v^*,w^*,e^*$ are as above,
and we denote by $\tilde{N}(e^*)$ the set of neighbours in $G$ of $e^*$
other than $v^*$ and $w^*$, then $|\tilde{N}(e^*)|\le\Delta^{1/4}$, each
vertex in $\tilde{N}(e^*)$ has degree at most $\Delta^{1/4}$ in $G$, and
each neighbour in $J'$ of $e^*$ is in $\tilde{N}(e^*)$ or is a neighbour in
$G$ of a vertex in $\tilde{N}(e^*)$.) Thus, the key to list colouring~$J$
will be to list colour~$L(H)$. Fortunately, list colouring line graphs is
much easier than list colouring arbitrary graphs (see e.g.\
\cite{Kahn00,Kay01,MoRe02}). In particular, using a sophisticated argument
due to Kahn~\cite{Kahn00}, we can prove the following lemma which specifies
certain sets of removable vertices which we can use to perform reductions.

Given a multigraph $H$ and sets $U$ and $W$ of vertices in $H$, we let
$e_H(U,W)$ denote the number of edges between $U$ and $W$, with any edge
between two vertices in $U \cap W$ counting twice. If the graph $H$ is
clear from the context, we may write just $e(U,W)$.

\begin{lemma}\label{lemnotover2}\mbox{}\\*
  For any $\eps>0$, there exists $\Delta_\eps$ such that the following
  holds for every graph $G$ with $\Delta=\Delta(G)\ge\Delta_\eps$.
  Suppose~$R$ is the core of a removable copy of~$H^*$ in~$G$, for some
  multigraph~$H$, such that for any set~$X$ of vertices of~$H$ and
  corresponding set~$X^*$ of vertices of the copy of~$H^*$,
  \[\sum_{x \in X^*}d_{G-R}(x)\:\le\:
  e_H(X,V(H)\sm X)+\tfrac1{30}\eps|X|\Delta.\]
  Then, given lists of size $\lstar$ for every vertex, any acceptable
  colouring on $G^2\!-\!R$ can be extended to an acceptable colouring
  of~$G^2$.
\end{lemma}

\noindent
The following lemma shows that we will indeed be able to find a removable
set of vertices which we can use to perform a reduction.

\begin{lemma}\label{lemeuler}\mbox{}\\*
  For any $\eps>0$, there exists~$\Delta_\eps$ such that every graph
  $G\in\MF$ with maximum degree at most $\Delta\ge\Delta_\eps$ contains at
  least one of the following:

  \smallskip
  \qitee{(a)}a removable vertex~$v$ which has degree less than
  $\frac32\Delta+\Delta^{1/2}$ in~$G^2$, or

  \smallskip
  \qitee{(b)}a removable copy of~$H^*$ with core~$R$, for some
  multigraph~$H$ which contains an edge and is such that for any set~$X$ of
  vertices of~$H$ and corresponding set $X^*$ of vertices of $H$,
  \[\sum_{x \in X^*}d_{G-R}(x)\:\le\:
  e_H(X,V(H)\sm X)+|X|\,\Delta^{9/10}.\]
\end{lemma}

\vspace{-4mm}\noindent
Combining Lemmas \ref{lemrem}, \ref{lemnotover2}, and~\ref{lemeuler} and
our observations on removing a removable vertex, yields Theorem~\ref{mt}
(with $o(1)$ replaced by $\eps$), provided that we choose $\Delta$ large
enough so that $3\Delta^{1/2}+2\le\eps\Delta$ (since then
$\frac32\Delta+\Delta^{1/2}<\lstar-1-2\Delta^{1/2}$) and
$\Delta^{9/10}\le\frac1{30}\eps\Delta$.

Thus, we need only prove the last two of these lemmas. The proof of
Lemma~\ref{lemeuler} is given in the next section. The proof of
Lemma~\ref{lemnotover2} is much more complicated and forms the bulk of the
paper. We follow the approach developed by Kahn~\cite{Kahn00} for his proof
that the list chromatic index of a multigraph is asymptotically equal to
its fractional chromatic number. We need to modify the proof so it can
handle our situation in which we have a graph which is slightly more than a
line graph and in which we have lists with fewer colours than Kahn
permitted. We defer any further discussion to Section~\ref{secnotover}.


\section{Proof of Lemma~\ref{lemeuler}: Finding a
  Reduction}\label{seceuler}

In this section we prove Lemma~\ref{lemeuler}. Throughout the section we
assume that~$\MF$ is a nice family of graphs. Since there is a $k$ such
that no graph in $\MF$ contains $K_{3,k}$ as a minor,
Theorem~\ref{adbounded} implies every graph in~$\MF$ has average degree at
most~$C_\MF$ for some constant~$C_\MF$.

Let~$G$ be a graph in~$\MF$ with vertex set~$V$ and maximum degree at
most~$\Delta$, and let $n=|V|$. We let~$B$ be the set of vertices of degree
exceeding~$\Delta^{1/4}$. Since the average degree of $G$ is at
most~$C_\MF$, we have $|B|<\dfrac{C_\MF\,n}{\Delta^{1/4}}$. Hence, another
application of Theorem \ref{K3nfree} implies that~$G$ contains a set~$R_0$
of at least $n-O\bfracb{n}{\Delta^{1/4}}$ removable vertices. We note that
if a vertex in~$R_0$ is adjacent to a vertex in~$B$ with degree less
than~$\half \Delta$, or is adjacent to at most one vertex in $B$, then its
total degree in the square $G^2$ is less than $\frac32\Delta+\Delta^{1/2}$
and conclusion~(a) of Lemma~\ref{lemeuler} holds. So, we can assume this is
not the case.

We let~$V_0$ be the set of vertices of~$G$ which have degree at
least~$\half \Delta$. Note that $V_0\subseteq B\subseteq V\sm R_0$. Since
every vertex in~$R_0$ has exactly two neighbours in~$V_0$, the sum of the
degrees of the vertices in~$V_0$ is at least $2 |R_0|$. This gives
$|V_0|\ge \dfrac{2|R_0|}{\Delta}\ge
\dfrac{2n}\Delta-O\bfracb{n}{\Delta^{5/4}}$.

We let~$S_0$ be the set of vertices in~$V_0$ which are adjacent to more
than~$\Delta^{7/8}$ vertices of $V\sm R_0$. Since every subgraph of $G$ has
average degree at most $C_\MF$, the total number of edges within $V\sm R_0$
is $O\bfracb{n}{\Delta^{1/4}}$. This implies that
$|S_0|=O\bfracb{n}{\Delta^{9/8}}$. We set $V_1=V_0\sm S_0$ and note that
$|V_1|\ge \dfrac{2n}\Delta-O\bfracb{n}{\Delta^{9/8}}$. We can conclude that
\begin{equation}\label{eq1}
  |V_1|\:\ge\:\frac{n}\Delta\quad \text{for large enough~$\Delta$}.
\end{equation}

We let~$R_1$ be the set of vertices in~$R_0$ adjacent to (exactly) two
vertices in~$V_1$. So every vertex in $R_0\sm R_1$ has one or two
neighbours in~$S_0$. By our bound on the size of~$S_0$, this means
$|R_0\sm R_1|\le |S_0|\,\Delta=O\bfracb{n}{\Delta^{1/8}}$ and hence
$|R_1|=n-O\bfracb{n}{\Delta^{1/8}}$. By our choice of~$S_0$ we have that
$e(V_1,V\sm R_0)\le\Delta^{7/8}|V_1|$. (Throughout this proof, $e(U,W)$
means $e_G(U,W)$.) Since every vertex in $R_0\sm R_1$ has at most one
neighbour in~$V_1$, we have
$e(V_1,R_0\sm R_1)\le |R_0\sm R_1|= O\bfracb{n}{\Delta^{1/8}}\le
O(\Delta^{7/8})|V_1|$, where the final inequality uses~\eqref{eq1}. We
obtain
\begin{equation}\label{eq666}
  e(V_1,V\sm R_1)\:=\: e(V_1,V\sm R_0)+e(V_1,R_0\sm R_1)\:\le\:
  O(\Delta^{7/8})|V_1|.
\end{equation}
We let~$F_1$ be the bipartite graph formed by the edges between the
vertices of~$R_1$ and the vertices of~$V_1$. We remind the reader that each
vertex of~$R_1$ has degree two in this graph. We let~$H_1$ be the
multigraph with vertex set~$V_1$ from which~$F_1$ is obtained by
subdividing each edge exactly once (so $F_1$ is a copy of $H^*_1$).

We check if~$F_1$ is a removable copy of~$H^*_1$ as in~(b). The only reason
that it might not be is that there is some subset $Z_1\subseteq V_1$ of
vertices of~$H_1$ such that:
\[e(Z_1,V\sm R_1)\:=\: \sum_{v\in Z_1}d_{G-R_1}(v)\:>\:
e_{H_1}(Z_1,V_1\sm Z_1)+|Z_1| \Delta^{9/10}.\]
In this case, we set $V_2=V_1\sm Z_1$, let~$R_2$ be the set of vertices
in~$R_1$ with no neighbours in~$Z_1$, let~$F_2$ be the bipartite subgraph
of~$G$ formed by the edges between the vertices of~$R_2$ and the vertices
of~$V_2$, and let~$H_2$ be the multigraph on~$V_2$ from which~$F_2$ is
obtained by subdividing each edge exactly once.

Now we check if~$F_2$ is empty or a removable copy of $H^*_2$ as in~(b). If
not, we can proceed in the same fashion, deleting a set $Z_2$ of vertices
from~$V_2$ and a set of vertices from~$R_2$, to obtain $V_3$,~$R_3$ and a
new bipartite graph $F_3$ with parts $V_3$ and $R_3$ and corresponding
multigraph $H_3$. We continue this process until it stops. We have
constructed new sets $V_1,R_1,V_2,R_2,\dots,\dots V_i,R_i$ such that
letting $Z_j=R_j-R_{j+1}$ for $j=1,\ldots,i-1$ we have:
\begin{equation}\label{eq222222}
  e(Z_j,V\sm R_j)\:=\: \sum_{v \in Z_j}d_{G-R_j}(v)\:>\:
  e_{H_j}(Z_j,V_j\sm Z_j)+|Z_j|\Delta^{9/10}.
\end{equation}

We must show that $R_i\ne\varnothing$, since then the corresponding
multigraph $H_i$ has at least one edge and we are done. To this end, we
note that for each $j=1,\ldots,i-1$
\begin{equation}\label{eq22667}
  e(V_{j+1},V\sm R_{j+1})\:=\: e(V_j\sm Z_j,V\sm R_{j+1})\:=\:
  e(V_j,V\sm R_j)-e(Z_j,V\sm R_j)+e(V_{j+1},R_j\sm R_{j+1}).
\end{equation}
Furthermore, for every vertex in $R_j\sm R_{j+1}$ adjacent to a vertex
in~$V_{j+1}$, there also is a vertex in~$Z_j$ it is adjacent to. Hence
$e(V_{j+1},R_j\sm R_{j+1})$ is precisely the number of edges of~$H_j$ with
exactly one endpoint in~$Z_j$: $e(V_{j+1},R_j\sm
R_{j+1})=e_{H_j}(Z_j,V_j\sm Z_j)$. Now~\eqref{eq222222} and \eqref{eq22667}
give:
\[e(V_j,V\sm R_j)\:>\: e(V_{j+1},V\sm R_{j+1})+|Z_j|\Delta^{9/10}.\]

Let $Z'=\bigcup_{j=1}^{i-1}Z_j= V_1\sm V_i$. Summing the inequality above
over $j=1,\ldots,i-1$ yields:
$e(V_1,V\sm R_1)\ge e(V_i,V\sm R_i)+|Z'|\Delta^{9/10}$. Using
\eqref{eq666}, this implies
\[|Z'|\:\le\: \frac{e(V_1,V\sm R_1)-e(V_i,V\sm R_i)}{\Delta^{9/10}}\:\le\:
\frac{O(\Delta^{7/8})\,|V_1|}{\Delta^{9/10}}\:=\:
|V_1|\,O(\Delta^{-1/40}).\]
Hence $|V_i|\ge|V_1|(1-O(\Delta^{-1/40}))$, which also gives
\begin{equation}\label{eq670}
  |V_1|\:\le\: (1+O(\Delta^{-1/40}))|V_i|.
\end{equation}
Since $V_i\subseteq V_1$, it follows from \eqref{eq666} and \eqref{eq670}
that:
\[e(V_i,V\sm R_0)\:\le\: e(V_1,V\sm R_1)\:\le\: O(\Delta^{7/8})|V_1|\:\le\:
O(\Delta^{7/8})|V_i|.\]
Finally, for each edge between~$V_i$ and $R_1\sm R_i$, we have an edge
between $R_1\sm R_i$ and~$Z'$ as well. We find
\[e(V_i,R_1\sm R_i)\:\le\: |Z'|\Delta\:\le\: |V_1|O(\Delta^{39/40})\:\le\:
O(\Delta^{39/40})|V_i|.\]
Combining these estimates we obtain
\[e(V_i,V\sm R_i)\:=\:e(V_i,V\sm R_0)+e(V_i,R_0\sm R_1)+e(V_i,R_1\sm R_i)\:
\le\:O(\Delta^{39/40})|V_i|.\]
But $V_i\ne\varnothing$, and each vertex in~$V_i$ has degree at
least~$\half \Delta$. This means that $e(V_i,R_i)>0$ for large
enough~$\Delta$. In particular, it follows that~$R_i$ is non-empty. Thus,
$H_i$ contains an edge. We have shown that~(b) holds.

This completes the proof of  Lemma~\ref{lemeuler}.
\proofend


\section{Proof of Lemma \ref{lemnotover2}: Reducing using Line
  Graphs}\label{secnotover}

It remains to prove Lemma \ref{lemnotover2}, which we do in this section. A
removable core satisfying the hypotheses of that lemma corresponds to the
edge set of a multigraph $H$ with maximum degree~$\Delta$, with each vertex
of the core corresponding to a distinct edge. Having coloured $G^2\!-\!R$
for such a removable core $R$, colouring the induced subgraph $G^2[R]$
translates to finding a list colouring of the line graph of $H$ (where an
edge inherits the list assigned to the corresponding vertex of~$H$) so that
certain side conditions are satisfied. Firstly, a vertex of $R$
corresponding to an edge $e=vw$ of $G$ may not use (a) a colour assigned to
one of its neighbours in the square of~$G$ lying in $V-R$, (b) a colour
within $\Delta^{1/4}$ of the colours assigned to its neighbours in $G$
lying in $V\sm R$. Secondly, two vertices of $R$ cannot (c) use the same
colour if they have a common neighbour in $G$ which does not correspond to
a vertex of $H$, or (d) use colours within $\Delta^{1/4}$ if they are
adjacent in $G$.

To handle side constraints (a) and (b) we simply delete the forbidden
colours from the list for $e$. Now, because the vertex corresponding to $e$
is removable, it has at most $\Delta^{1/4}$ neighbours which are not $v$ or
$w$, and each such neighbour has degree at most $\Delta^{1/4}$. So even
after these deletions, the list for $e$ will have at least $\lstar_e$
elements, where
\begin{equation}\label{eqn.lstaredef}
  \lstar_e\:=\:
  \bigl\lceil\bigl(\tfrac32+\eps\bigr)\Delta-
  (\Delta-d_H(v))-(\Delta-d_H(w))-3\Delta^{1/2}\bigr\rceil.
\end{equation}

To handle side constraints (c) and (d) we use two auxiliary graphs $J_1$
and $J_2$. The removability of the vertices of $R$ ensures that these
graphs have bounded degree (in terms of $\Delta$). Thus, as we show later,
if $G$ is $\Delta$-regular then Lemma~\ref{lemnotover2} follows directly
from the following result.

\begin{lemma}\label{lemfec}\mbox{}\\*
  For every $0<\eps<\frac14$ there is a~$\Delta_\eps$ such that the
  following holds for all $\Delta\ge\Delta_\eps$. Let~$H$ be a multigraph
  with vertex set~$V$ and maximum degree at most~$\Delta$. For every edge
  $e$, let $L(e)$ be a list of colours. Let $J_1$ be a graph with vertex
  set~$E(H)$ and maximum degree at most~$\Delta^{1/2}$, and let~$J_2$ be a
  graph with vertex set~$E(H)$ and maximum degree at most~$\Delta^{1/4}$.
  Suppose that the following two conditions are satisfied.{

    \smallskip
    \qitee{(1)}For every edge~$e$: $|L(e)|=\lstar_e$.

    \smallskip
    \qitee{(2)}For every set~$X$ of an odd number of vertices of~$H$:
    \[\sum_{v\in X}(\Delta-d(v))-e(X,V\sm X)\:\le\:
    \tfrac1{30}\eps|X|\Delta.\]}%
  Then we can find an acceptable edge-colouring of $H$ such that any pair
  of edges of~$H$ joined by an edge of~$J_1$ receive different colours, and
  any pair of edges of~$H$ joined by an edge of~$J_2$ receive colours that
  differ by at least~$\Delta^{1/4}$.
\end{lemma}

\begin{remark}\label{rem.tildec1}\quad\rm
  Condition~(2) of Lemma~\ref{lemfec} applied to the set $X=\{v\}$ implies
  that for each vertex~$v$,
  $d(v)\ge\bigl(\half-\frac1{60}\eps\bigr)\Delta$. Hence, for each edge~$e$
  we have
  $\lstar_e\ge\half\Delta+\frac{29}{30}\eps\Delta-3\Delta^{1/2}>
  \half\Delta$ if~$\Delta$ is sufficiently large.
\end{remark}

\noindent
As we show at the end of this section, a relatively simple trick allows us
to reduce to the case when $G$ is regular. So the bulk of the work in the
section is in proving Lemma \ref{lemfec}: we complete this task at the end
of Subsection~\ref{subsec.modkahn}.

The way we prove Lemma \ref{lemfec} is by exploiting some beautiful work of
Kahn, developed to show that, letting $\chi'_f(H)$ be the fractional
chromatic index of $H$, we have that the list chromatic index of $H$ is
$(1+o(1)) \chi'_f(H)$. We will do two things: (i) explain why Lemma
\ref{lemfec} follows from Kahn's proof in the special case when~$J_1$ and
$J_2$ are empty (that is, $J_1$ and $J_2$ have no edges, and so are
irrelevant), and (ii) discuss the modifications needed to Kahn's proof to
deal with~$J_1$ and $J_2$.

Kahn's proof analyses an iterative procedure which in each iteration, for
each colour $\gamma$, randomly extends a matching in the spanning subgraph
$H_\gamma$ of $H$ whose edges are those on which $\gamma$ is available, to
progressively colour more and more of the edges of $H$. The first step of
his proof is to show that if each list has at least $(1+o(1))\chi'_f(H)$
elements, then there is a probability distribution on these matchings which
ensures that: (a)~for every edge $e$ and colour~$\gamma$ in $L(e)$, the
probability that $e$ is in the matching of colour $\gamma$ is
$|L(e)|^{-1}$, and (b) other desirable properties hold. The second step is
to show that for any family of lists for which there are probability
distributions satisfying (a) and (b), this iterative procedure yields a
colouring of~$E(H)$ where each edge gets a colour from its own list.

It is natural to state Kahn's result precisely, before discussing our
modification of it. Having done so, before delving into the details of
Kahn's proof, we will show that for lists satisfying the conditions of
Lemma~\ref{lemfec} there are probability distributions on the matchings
satisfying (a) and~(b) above. We can then apply Kahn's work, as a black
box, to prove Lemma~\ref{lemfec} in the special case when $J_1$ and $J_2$
are empty.

We then turn to strengthening the result so that it can deal with $J_1$ and
$J_2$. This has two parts. First we perform some straightforward
preprocessing which allows us to reserve some colours which can be used to
recolour vertices involved in conflicts caused by $J_1$ and $J_2$. Then we
impose additional constraints which provide, for each iteration, an upper
bound on the number of edges incident to each vertex which are involved in
such conflicts because of a colour they are assigned in that iteration.
Here we must get into the guts of Kahn's proof sufficiently, so as to be
able to explain the (relatively straightforward) additions to it which
allow us to do so. Then in a postprocessing phase, we recolour to eliminate
such conflicts using the colours we reserved in the first phase.

We will actually discuss this preprocessing and postprocessing first. We do
this in part because all of the rest of the discussion involves Kahn's
proof, while the pre/postprocessing does not, so it is natural to hive it
off; and in part because our discussion of the preprocessing introduces the
Lov\'asz Local Lemma, an important tool in Kahn's proof, in a simple
setting.


\subsection{Before and After}

Our preprocessing consists of applying the following lemma (which we prove
in this subsection).

\begin{lemma}\label{beforeandafter}\mbox{}\\*
  Suppose we are given a multigraph $H$ with maximum degree $\Delta$
  sufficiently large, satisfying the conditions of Lemma \ref{lemfec}. Then
  we can find, for every list $L(e)$, two disjoint sublists $L'(e)$
  and~$R(e)$ such that:

  \smallskip
  \qitee{(a)}no colour in $R(e)$ appears in $L'(f)$ for any edge $f$
  incident to $e$ in $H$;

  \smallskip
  \qitee{(b)}$|L'(e)|\ge|L(e)|-\frac23\eps\Delta$; and

  \smallskip
  \qitee{(c)}$|R(e)|\ge\Delta^{9/10}$.
\end{lemma}

\noindent
We then apply the following variant (to be proved later) of Lemma
\ref{lemfec} to the family $L'(e)$ of lists, using $\half\eps$ in place of
$\eps$. In this variant, condition~(2) is weakened by replacing 30 by 10,
and the conclusions are weakened by allowing some of the graph to remain
uncoloured. We can apply this amended lemma because of our bound on
$|L(e)-L'(e)|$ and the fact that conditions~(1) and~(2) of Lemma
\ref{lemfec} hold before the preprocessing.

\begin{lemma}\label{lemfec2}\mbox{}\\*
  For every $0<\eps<\frac14$ there is a~$\Delta_\eps$ such that the
  following holds for all $\Delta\ge\Delta_\eps$. Let~$H$ be a multigraph
  with vertex set~$V$ and maximum degree at most~$\Delta$. For every edge
  $e$ we are given a list~$L(e)$ of acceptable colours. Additionally, $J_1$
  is a graph with vertex set~$E(H)$ and maximum degree at
  most~$\Delta^{1/2}$ and~$J_2$ is a graph with vertex set~$E(H)$ and
  maximum degree at most~$\Delta^{1/4}$. Suppose that the following two
  conditions are satisfied.{

    \smallskip
    \qitee{(1)}For every edge~$e$: $|L(e)|=\lstar_e$.
    \smallskip

    \qitee{(2)}For every set~$X$ of an odd number of vertices of~$H$:
    \[\sum_{v\in X}(\Delta-d(v))-e(X,V\sm X)\:\le\:
    \tfrac1{10}\eps|X|\Delta.\]}%
  Then we can find an acceptable edge-colouring of $H$ such that by
  uncolouring a set of edges of~$H$, including at most
  $\frac13\Delta^{9/10}$ edges incident to any vertex $v$ of $H$, we obtain
  a partial colouring such that any two coloured edges joined by an edge
  of~$J_1$ receive different colours, and any two coloured edges joined by
  an edge of~$J_2$ receive colours that differ by at least~$\Delta^{1/4}$.
\end{lemma}

\begin{remark}\label{rem.lstare2}\quad\rm
  Much as in Remark~\ref{rem.tildec1}, condition~(2) in Lemma~\ref{lemfec2}
  gives $d(v)\ge\bigl(\half-\frac1{20}\eps\bigr)\Delta$, and so
  $\lstar_{e}>\half\Delta$ for each edge $e$, for sufficiently large
  $\Delta$.
\end{remark}

\noindent
To complete the proof of Lemma \ref{lemfec} (assuming
Lemmas~\ref{beforeandafter} and~\ref{lemfec2}), we uncolour edges as
specified in Lemma~\ref{lemfec2}, and then recolour each such edge $e$
using a colour from its reserve list~$R(e)$. By conclusion (a) of
Lemma~\ref{beforeandafter}, this colour cannot conflict with the colour of
any edge incident to~$e$ which was not uncoloured. So, in colouring $e$ we
must avoid any colour from~$R(e)$ assigned to an edge incident to it which
we have uncoloured (and re-coloured), avoid any colour assigned to a
neighbour in $J_1$, and avoid any colour within $\Delta^{1/4}$ of
neighbours in $J_2$. But in total there are at most
$3\Delta^{1/2}+\frac23\Delta^{9/10}$ colours to avoid, so if $\Delta $ is
large enough we can carry out the recolouring greedily.

So, to prove Lemma \ref{lemfec} it is enough to prove
Lemmas~\ref{beforeandafter} and~\ref{lemfec2}. In the remainder of this
section we prove Lemma~\ref{beforeandafter}; the proof of
Lemma~\ref{lemfec2} will not be complete until the end of
Section~\ref{subsec.modkahn}.

The key to the proof Lemma~\ref{beforeandafter} is the following general
lemma.

\begin{lemma}[Erd\H{o}s and Lov\'asz \cite{ErLo75}]{\bf (Local
    Lemma)}\label{lemlocal} \mbox{}\\*
  Suppose that~$\mathcal{B}$ is a set of (bad) events in a probability
  space~$\Omega$. Suppose further that there are~$p$ and~$d$ such that:

  \smallskip
  \qitee{(1)}for every event~$B$ in~$\mathcal{B}$, there is a
  subset~$\mathcal{S}_B$ of~$\mathcal{B}$ of size at most~$d$, such that
  the conditional probability of~$B$, given any conjunction of occurrences
  or non-occurrences of events in $\mathcal{B}\sm\mathcal{S}_B$, is at
  most~$p$, and

  \smallskip
  \qitee{(2)}$\rme pd<1$.

  \smallskip\noindent
  Then with positive probability, none of the events in~$\mathcal{B}$
  occur.
\end{lemma}

\noindent
In our preprocessing step, we apply the Local Lemma to the (product)
probability space obtained by, for each colour $c$ and vertex $v$,
independently assigning $c$ to a list $R(v)$ with
probability~$\frac16\eps$.

For an edge $e$ with endpoints $u$ and $v$, we set
$R(e)=L(e)\cap R(u)\cap R(v)$ and $L'(e)=L(e)-(R(u)\cup R(v))$. Note that
the sublists $R(e)$ and $L'(e)$ defined in this way must satisfy
condition~(a) of Lemma~\ref{beforeandafter}. We shall now prove that, with
positive probability, conditions (b) and (c) are also satisfied for all
edges $e$. Let $B_e$ be the event that condition (b) is not satisfied for
the edge $e$, i.e.\
$|L(e)\sm L'(e)|=|L(e)\cap(R(u)\cup R(v))|>\frac23\eps\Delta$. Let $C_e$ be
the event that $|R(e)|<\frac1{100}\eps^2\Delta$ for the edge $e$. For
sufficiently large $\Delta$, $\frac1{100}\eps^2\Delta\ge \Delta^{9/10}$, so
if $C_e$ does not hold then condition (c) is satisfied for edge $e$.

To conclude the proof of Lemma~\ref{beforeandafter}, we apply the Local
Lemma to show that with positive probability none of the bad events
$B_e,C_e$ occurs. Now $B_e$ and $C_e$ are determined completely by the
random assignments made at the endpoints of $e$, so letting
$S_{B_e}=S_{C_e}=\{B_f,C_f\mid\text{$f=e$ or $f$ is incident to $e$}\}$, we
see that condition (1) of the Local Lemma holds with $d=4\Delta-2$ and $p$
the maximum of the unconditional probability of~$B_e$ and the unconditional
probability of~$C_e$.

Now, for any edge $e$ with endpoints $u$ and $v$, the number of colours in
$R(e)=L(e)\cap R(u)\cap R(v)$ is the sum of $|L(e)|$ independent 0-1
variables, each of which is 1 with probability $\frac1{36}\eps^2$. So, the
expected value of this random variable is $\frac1{36}\eps^2\lstar_e$; and
by Remark~\ref{rem.lstare2}, for large $\Delta$, this is at least
$\frac1{72}\eps^2\Delta$. Standard concentration inequalities (e.g.\ the
Chernoff bounds) tell us that the probability that this variable differs
from its expected value by some $t>0$ which is less than its expected value
is $2^{-\Omega(t^2/\Delta)}$. So, the probability of $C_e$ is
$2^{-\Omega(\Delta)}$.

In the same vein, for any edge $e$ with endpoints $u$ and $v$, the number
of colours in $L(e)\cap(R(u)\cup R(v))$ is the sum of $|L(e)|$ independent
0-1 variables, each of which is 1 with probability at most $\frac26\eps$.
We obtain that the expected value of this random variable is at most
$\frac13\eps\lstar_e$, which is at most
$\frac13\eps\bigl(\frac32+\eps\bigr)\Delta<\frac7{12}\eps\Delta$. Again
applying standard concentration inequalities, we see that the probability
of~$B_e$ is $2^{-\Omega(\Delta)}$.

Thus for large $\Delta$ the hypotheses of the Local Lemma hold with
$p=\sfrac{1}{3d}$, and we have completed the proof of
Lemma~\ref{beforeandafter}.
\proofend


\subsection{Kahn's Result as a Black Box}

Kahn presents an algorithm in~\cite{Kahn00} which shows that the list
chromatic index of a multigraph exceeds its fractional chromatic index
by~$o(\Delta)$. Actually, the algorithm implicitly contains a subroutine
which does more than this, providing a proof (which we shall describe
later, see Subsection~\ref{proof.kahnth2}) of the following result.

\begin{theorem}[Kahn \cite{Kahn00}]\label{kahnth2}\mbox{}\\*
  For every~$\delta$ with $0<\delta<1$ and every $C>0$, there exists
  a~$\Delta_{\delta,C}$ such that the following holds for all
  $\Delta\ge\Delta_{\delta,C}$. Let~$H$ be a multigraph with maximum degree
  at most~$\Delta$, and with a list~$L(e)$ of acceptable colours for every
  edge~$e$. Suppose that the following conditions are satisfied:

  \smallskip
  \qitee{(1)} For every vertex $v$ and edge $e$ incident to $v$,
  $|L(e)|\ge d(v)(1+\delta)$.

  \smallskip
  \qitee{(2)} For every odd set $X$ of vertices of $H$, the sum of
  $z_e=\dfrac{1+\delta}{|L(e)|}$ over the edges joining vertices of $X$ is
  at most $\half(|X|-1)$.

  \smallskip
  \qitee{(3)}For every edge~$e$: $|L(e)|\ge\sfrac{\Delta}{C}$.

  \smallskip\noindent
  Then we can find an acceptable edge-colouring of $H$.
\end{theorem}

\noindent
This theorem is not explicitly stated in Kahn's paper, although it follows
in just a few pages from the proof of \cite[Lemma~3.1]{Kahn00}, which forms
the bulk of his paper. We pull the result out of his discussion, after
showing that Theorem \ref{kahnth2} implies the special case of
Lemma~\ref{lemfec2} when~$J_1$ and $J_2$ are empty (and we have just the
simple conclusion ``\emph{Then we can find an acceptable edge-colouring of
  $H$.}'').

To prove this implication, we need only show that for $0<\eps<\frac14$ and
sufficiently large $\Delta$, any family of lists satisfying conditions (1)
and (2) of Lemma~\ref{lemfec2} must satisfy conditions (1) and (2) of
Theorem \ref{kahnth2} for $\delta=\half\eps$, since we noted in
Remark~\ref{rem.lstare2} that, for sufficiently large $\Delta$,
$\lstar_e>\half\Delta$ for each edge $e$, so condition (3) holds with
$C=2$. This is the content of the following lemma.

\begin{lemma}\label{lemkey}\mbox{}\\*
  Let $0<\eps<\frac14$. Then there is a~$\Delta_\eps$ such that for every
  $\Delta\ge\Delta_\eps$ the following holds. Let~$H$ be a multigraph with
  vertex set~$V$ and maximum degree at most~$\Delta$, and for every
  edge~$e$ let~$L(e)$ be a list of acceptable colours. Suppose the
  following conditions are satisfied:

  \smallskip
  \qitee{(1)}For every edge~$e$: $|L(e)|\ge\lstare$.

  \smallskip
  \qitee{(2)}For every set~$X$ of an odd number of vertices of~$H$:
  \[\sum_{v\in X}(\Delta-d(v))-e(X,V\sm X)\:\le\:
  \tfrac1{10}\eps|X|\Delta.\]

  \noindent
  Then, the following properties hold:

  \smallskip
  \qitee{(a)} For every vertex $v$ and edge $e$ incident to $v$:
  $|L(e)|\ge\bigl(1+\half\eps\bigr)d(v)$.

  \smallskip
  \qitee{(b)} For every odd set $X$ of vertices of $H$: the sum of
  $z_e=\dfrac{1+\half\eps}{|L(e)|}$ over the edges joining vertices of $X$
  is at most $\half(|X|-1)$.
\end{lemma}

\begin{proof}
  Whenever an inequality requires~$\Delta$ to be large enough, we use
  ``$\ge_*$''.

  To begin, we note that condition (2) of the lemma implies that every
  vertex~$w$ of~$H$ has degree at least
  $\bigl(\half-\frac1{20}\eps\bigr)\Delta$. Hence, for any edge $e=vw$
  of~$H$, condition (1) of the lemma implies that
  $|L(e)|\ge d(v)+\frac{19}{20}\eps\Delta-3\Delta^{1/2}\ge_*
  d(v)+\frac34\eps\Delta\ge\bigl(1+\frac34\eps\bigr)d(v)$. Thus property
  (a) holds.

  \smallskip
  Now we check property (b) for the case $|X|=3$. Consider a subgraph~$F$
  of~$H$ with vertex set~$X$ consisting of three distinct vertices $x,y,z$,
  and with $\alpha\Delta >0$ edges. Note that $\alpha\le\frac32$, since
  $2|E(F)|\le d(x)+d(y)+d(z)\le 3\Delta$ (where $d(\,\cdot\,)$ refers to
  the degree in the whole graph~$H$). Applying the second condition gives
  \[3\Delta-d(x)-d(y)-d(z)\:\le\: e(X,V\sm X)+\tfrac3{10}\eps\Delta.\]
  Since we also have
  $3\Delta-d(x)-d(y)-d(z)=3\Delta-2\alpha\Delta-e(X,V\sm X)$, we obtain
  \[3\Delta-d(x)-d(y)-d(z)\:\le\:
  \Bigl(\half\bigl(3+\tfrac3{10}\eps\bigr)-\alpha\Bigr)\Delta,\]
  which we can rewrite as
  \[\tfrac32\Delta\:\ge\:
  3\Delta-d(x)-d(y)-d(z)-\tfrac3{20}\eps\Delta+\alpha\Delta.\]
  Substituting this into the first condition of the lemma yields that for
  any edge $e=uv$ in~$F$:
  \[|L(e)|\:\ge\:\Delta+(d(u)+d(v)-d(x)-d(y)-d(z))+
  \bigl(\alpha+\tfrac{17}{20}\eps\bigr)\Delta-3\Delta^{1/2}.\]
  Since $\Delta-d(w)$ is non-negative for any~$w$ in~$X$, and
  $\{u,v\}\subset\{x,y,z\}$, this yields
  \[|L(e)|\:\ge\:
  \bigl(\alpha+\tfrac{17}{20}\eps\bigr)\Delta-3\Delta^{1/2}\:\ge_*\:
  \bigl(\alpha+\tfrac34\eps\bigr)\Delta.\]
  Since $\alpha\le\frac32$, this gives that for any edge~$e$ in~$F$,
  $z_e\le\dfrac{1+\half\eps}{\bigl(\alpha+\frac34\eps\bigr)\Delta}\le
  \dfrac1{\alpha\Delta}$.

  We can conclude that
  $\sum_{e\in E(F)}z_e\le (\alpha\Delta)\cdot\dfrac1{\alpha\Delta}=1$. This
  shows that property (b) holds for all sets~$X$ of three vertices.

  \smallskip
  Given a multigraph $G$ and a vertex $v$, we write $E_{G,v}$, or simply
  $E_v$, for the set of edges incident to $v$. Next consider any
  subgraph~$F$ of~$H$ with vertex set~$X$, where $|X|\ge5$ is odd.
  Throughout the rest of the proof, for each vertex $v$ of $F$ we write
  $E_v$ for $E_{F,v}$. We partition the vertices of~$F$ into a set~$B$ of
  vertices with degree at least~$\frac34\Delta$ and a set~$S$ of vertices
  with degree less than~$\frac34\Delta$ (where degrees are in $H$).

  \medskip\noindent
  \textbf{Case 1}: There is a vertex in~$B$ with degree at
  most~$\frac78\Delta$, or a vertex in~$S$ with degree at
  most~$\frac58\Delta$.

  \smallskip
  For any edge $e=vw$ with $w \in B$, applying the first condition of the
  lemma, we obtain
  $|L(e)|\ge d(v)+\frac14\Delta+\eps\Delta-3\Delta^{1/2}\ge_*
  \bigl(1+\half\eps\bigr)\frac54d(v)$. Thus, $z_e\le\dfrac{4}{5d(v)}$.
  Since (a) holds, we have that $z_e\le\sfrac{1}{d(v)}$ for all edges $e$
  incident to $v$, and hence for each vertex $v\in B$:
  \[\sum_{e\in E_v}z_e\:\le\:
  \frac4{5d(v)}|E_v|+\frac1{5d(v)}e(\{v\},S);\]
  while for each vertex~$v$ in~$S$:
  \[\sum_{e\in E_v}z_e\:\le\:
  \frac1{d(v)}|E_v|-\frac1{5d(v)}e(\{v\},B).\]

  We estimate, using that the vertices in~$S$ have smaller degree than the
  vertices in~$B$,
  \begin{align*}
    2\sum_{e\in E(F)}z_e\:& = \:\sum_{v\in X}\sum_{e\in E_v}z_e\\
    &\le\:\sum_{v\in B}\frac4{5d(v)}|E_v|+
    \sum_{v\in S}\frac1{d(v)}|E_v|+\!\!
    \sum_{\substack{e\in E(F)\\[1pt]e=vw,\:v\in B,\:w\in S}}\!\!
    \Bigl(\frac1{5d(v)}-\frac1{5d(w)}\Bigr)\\
    &\le\:\sum_{v\in B}\frac4{5d(v)}|E_v|+
    \sum_{v\in S}\frac1{d(v)}|E_v|\\
    &\le\:\tfrac45|B|+|S|-\tfrac45e(X,V\sm X)\frac1\Delta.
  \end{align*}
  Also, applying the second condition of the lemma and the assumption for
  this Case~1, we see that
  \[e(X,V\sm X)\:\ge\:
  \tfrac14\Delta|S|+\tfrac18\Delta-\tfrac1{10}\eps|X|\Delta.\]
  Combining the two estimates, we obtain
  \[2\sum_{e\in E(F)}z_e\:\le\:
  \tfrac45|B|+\tfrac45|S|-\tfrac1{10}+\tfrac2{25}\eps|X|\:=\:
  |X|\bigl(\tfrac45+\tfrac2{25}\eps\bigr)-\tfrac1{10}.\]
  Since $\eps\le\frac14$ and $|X|\ge5$, this yields that
  $2\!\sum\limits_{e\in E(F)}\!z_e\le|X|-\frac{9}{50}|X|-\frac1{10}\le
  |X|-1$, as required for property~(b).

  \medskip\noindent
  \textbf{Case 2}: Every vertex in~$B$ has degree at least~$\frac78\Delta$
  and every vertex in~$S$ has degree at least~$\frac58\Delta$.

  \smallskip
  Applying the first condition of the lemma as in Case~1, we see that for
  an edge~$e$ with endvertices $v$ and $w$, we have
  $|L(e)|\ge d(v)+\frac18\Delta+\eps\Delta-3\Delta^{1/2}\ge_*
  \bigl(1+\half\eps\bigr)\cdot\frac98d(v)$, and if $v\in B$, then we get
  $|L(e)|\ge d(v)+\frac38\Delta+\eps\Delta-3\Delta^{1/2}\ge_*
  \bigl(1+\half\eps\bigr)\cdot\frac{11}8d(v)$. So, for each vertex $v\in B$
  we have
  \[\sum_{e\in E_v}z_e\:\le\:
  \frac8{11d(v)}|E_v|+\frac{16}{99d(v)}e(\{v\},S);\]
  while for each vertex~$v$ in~$S$ we can write
  \[\sum_{e\in E_v}z_e\:\le\:
  \frac8{9d(v)}|E_v|-\frac{16}{99d(v)}e(\{v\},B).\]
  Following the same method as in Case~1, this leads to
  \begin{align*}
    2\sum_{e\in E(F)}z_e\:=\: \sum_{v\in X}\sum_{e\in E_v}z_e\:&
    \le\:\sum_{v\in B}\frac8{11d(v)}|E_v|+
    \sum_{v\in S}\frac8{9d(v)}|E_v|\\
    &\le\:\tfrac8{11}|B|+\tfrac89|S|-\tfrac8{11}e(X,V\sm X)\frac1\Delta.
  \end{align*}
  Also, applying the second condition of the lemma, we see that
  \[e(X,V\sm X)\:\ge\:\tfrac14\Delta|S|-\tfrac1{10}\eps|X|\Delta.\]
  Since $\frac89 -\frac2{11}<\frac8{11}$, we obtain
  \[2\sum_{e\in E(F)}z_e\:\le\:
  |X|\bigl(\tfrac8{11}+\tfrac8{110}\eps\bigr).\]
  Since $\eps<1$ and $|X|\ge5$, this yields that
  $2\!\sum\limits_{e\in E(F)}\!z_e\le\frac45|X|\le|X|-1$, as required.
\end{proof}


\subsection{Opening the Lid}

In this section, we discuss how Theorem \ref{kahnth2} is implicitly proved
in Kahn's paper. First however, we need to introduce the special type of
probability distributions he considers.

\subsubsection{Hard-core Probability Distributions on Matchings}

For a probability distribution~$p$, defined on the matchings of a
multigraph~$H$, we let~$x^p(e)$ be the probability that~$e$ is in a
matching chosen according to~$p$. We call the value of~$x^p(e)$ the
\emph{marginal of~$p$ at~$e$}. The vector~$x^p=(x^p(e))$ indexed by the
edges~$e$ is called the \emph{marginal of~$p$}.

We are actually interested in using special types of probability
distributions on the matchings of~$H$. A probability distribution~$p$ on
the matchings of~$H$ is \emph{hard-core} if it is obtained by associating a
non-negative real~$\lambda^p(e)$ to each edge~$e$ of~$H$ so that the
probability that we pick a matching~$M$ is proportional to
$\prod_{e\in M}\lambda^p(e)$. I.e.\ setting
$\lambda^p(M)=\prod_{e\in M}\lambda^p(e)$ and letting~$\mathcal{M}(H)$ be
the set of matchings of~$H$, we have
\[p(M)\:=\:
\frac{\lambda^p(M)}{\!\sum\limits_{N\in\mathcal{M}(H)}\!\lambda^p(N)}.\]
We call the values~$\lambda^p(e)$ the \emph{activities of~$p$}.

We want to characterise for which families of lists, a multigraph $H$ with
maximum degree at most $\Delta$ has a hard-core probability distribution
$p$ on its matchings such that we have (i) $x^p(e)=|L(e)|^{-1}$ for each
edge~$e$, and (ii)~for some $K>0$, $\lambda^p(e)\le\sfrac{K}{\Delta}$ for
each edge~$e$.

Finding an arbitrary probability distribution on the matchings of~$H$ with
marginals~$x$ is equivalent to expressing~$x$ as a convex combination of
incidence vectors of matchings of~$H$. So, we can use a seminal result due
to Edmonds~\cite{Edm65} to understand for which~$x$ this is possible.

The \emph{matching polytope}~$\MP(H)$ is the set of non-negative
vectors~$x$ indexed by the edges of~$H$ which are convex combinations of
incidence vectors of matchings.

\begin{theorem}[Edmonds~\cite{Edm65}]{\bf(Characterisation of the Matching
    Polytope)}\label{mpthm}\mbox{}\\*
  For a multigraph~$H$, a non-negative vector $x=(x_e:e\in E(H))$ is
  in~$\MP(H)$ if and only if

  \smallskip
  \qitee{(1)}for every vertex~$v$ of~$H$:
  $\sum\limits_{e\in E_v}\!x_e\le1$, and

  \smallskip
  \qitee{(2)} for every set $X$ of vertices of~$H$ with $|X|\ge3$ and odd:
  $\sum\limits_{e\in E(X)}\!x_e\le\half(|X|-1)$.
\end{theorem}

\begin{remark}\quad
  \rm It is easy to see that conditions~(1) and~(2) are necessary as they
  are satisfied by all the incidence vectors of matchings and hence by all
  their convex combinations. It is the fact that they are sufficient which
  makes the theorem so valuable.
\end{remark}

\noindent
It turns out that we can choose a hard-core distribution with marginals~$x$
provided all of the above inequalities are strict.

\begin{lemma}[Lee~\cite{Lee90}; Rabinovitch, Sinclair and
  Widgerson~\cite{RSW92}]\label{lemhcd7}\mbox{}\\*
  For a multigraph~$H$, there is a hard-core distribution with marginals a
  given non-negative vector $x=(x_e:e\in E(H))$ if and only if

  \smallskip
  \qitee{(a)} for every vertex $v$ of $H$:
  $\sum\limits_{e\in E_v}\!x_e<1$, and

  \smallskip
  \qitee{(b)} for every set $X$ of vertices of~$H$ with $|X|\ge3$ and odd:
  $\sum\limits_{e\in E(X)}\!x_e<\half(|X|-1)$.
\end{lemma}

\noindent
In order to ensure that the~$\lambda^p$ are bounded, it turns out that we
just have to bound our distance from the boundary of the Matching Polytope.

\begin{lemma}[Kahn and Kayll~\cite{KaKa97}]\label{hcd1}\mbox{}\\*
  For all~$\delta$ with $0<\delta<1$, there is a~$\beta$ such that, for
  every multigraph~$H$, if~$p$ is a hard-core distribution whose marginals
  are in $(1-\delta)\MP(H)$, then

  \smallskip
  \qitee{(a)}for every edge~$e$ of~$H$: $\lambda^p(e)<\beta x^p(e)$, and

  \smallskip
  \qitee{(b)}for every vertex~$v$ of~$H$:
  $\sum\limits_{e\in E_v}\!\lambda^p(e)<\beta$.
\end{lemma}

\noindent
The material presented in this subsection is discussed in fuller detail in
\cite[Chapter~22]{MoRe02}.

\subsubsection{The Proof of Theorem \ref{kahnth2}}\label{proof.kahnth2}

As we are about to show, we can prove Theorem \ref{kahnth2} by combining
the results of the last section with the following result, which is also
implicit in \cite{Kahn00}, but much easier to pull out of it. Recall that
for a multigraph $H$ with a list $L(e)$ of acceptable colours for every
edge~$e$, for each colour~$\gamma$, we let $H_\gamma$ be the spanning
subgraph of~$H$ with edges those $e$ such that $\gamma\in L(e)$.

\begin{theorem}[Kahn \cite{Kahn00}]\label{kahnth}\mbox{}\\*
  For every~$\delta$ with $0<\delta<1$ and every $K>0$, there exists
  a~$\Delta_{\delta,K}$ such that the following holds for all
  $\Delta\ge\Delta_{\delta,K}$. Let~$H$ be a multigraph with maximum degree
  at most~$\Delta$, and with a list~$L(e)$ of acceptable colours for every
  edge~$e$.

  Suppose that for every colour~$\gamma$ there exists a hard-core
  distribution~$p_\gamma$ on the matchings of~$H_\gamma$, with
  corresponding marginal~$x^{p_\gamma}$ on the edges, satisfying the
  following conditions:{

    \smallskip
    \qitee{(1)}For every edge~$e$:
    $\sum\limits_{\gamma\in L(e)}\!x^{p_\gamma}(e)=1$.

    \qitee{(2)}For every edge~$e$ and colour~$\gamma$:
    $\lambda^{p_\gamma}(e)\le\sfrac{K}{\Delta}$.}

  \smallskip\noindent
  Then we can find an acceptable edge-colouring of~$H$.
\end{theorem}

\begin{proof}{\bf\!\!of Theorem \ref{kahnth2}, assuming Theorem
    \ref{kahnth}} \ Consider a multigraph $H$ and family of lists
  satisfying the conditions of Theorem \ref{kahnth2}. For each colour
  $\gamma$ consider the vector $x^{p_\gamma}$ indexed by the edges of
  $H_\gamma$, where $x^{p_\gamma}_e=|L(e)|^{-1}$. Then condition (1) of
  Theorem~\ref{kahnth} holds. Conditions~(1) and (2) of
  Theorem~\ref{kahnth2}, combined with Edmonds's characterisation of the
  matching polytope, tell us that $x^{p_\gamma}$ is in
  $(1-\delta)\MP(H_{\gamma})$. Now let $\beta$ be as in Lemma \ref{hcd1}.
  Then there is a hard-core distribution on $H_\gamma$ with marginals
  $x^{p_\gamma}$ such that, for every edge~$e$ and colour~$\gamma$, we have
  $\lambda^{p_\gamma}(e)\le\beta x^{p_\gamma}(e)$. Thus, setting $K=\beta
  C$, by condition (3) of Theorem \ref{kahnth2}, for every edge~$e$ and
  colour~$\gamma$, we have $\lambda^{p_\gamma}(e)\le\sfrac{K}{\Delta}$.
  Hence condition (2) of Theorem~\ref{kahnth} holds, and we can apply that
  result to complete the proof.
\end{proof}

\noindent
The proof of Kahn's main theorem, \cite[Theorem~1.1]{Kahn00}, demonstrates
that we can obtain an acceptable edge-colouring for a given family of lists
on the edges of a multigraph $H$ with maximum degree at most~$\Delta$, by
first showing that there are hard-core distributions with marginals
$|L(e)|^{-1}$ in each~$H_\gamma$ which satisfy the hypotheses of
\cite[Lemma~3.1]{Kahn00} (this is done in the second paragraph of
\cite[page~127]{Kahn00}), and then iteratively applying this lemma to reach
a situation where we can finish off greedily.

To prove Theorem~\ref{kahnth} following exactly the same scheme, we need
simply ensure that hard-core distributions satisfying the hypotheses of
\cite[Lemma~3.1]{Kahn00} with marginals $|L(e)|^{-1}$ at~$e$ exist for our
family of lists. But the hypotheses of \cite[Lemma~3.1]{Kahn00} are
precisely that conditions~(1) and (2) of Theorem~\ref{kahnth} hold, and
thus we have established Theorem~\ref{kahnth}.


\subsection{Modifying Kahn's Result}\label{subsec.modkahn}

In this section, we will modify Kahn's result so that by taking $J_1$ and
$J_2$ into account it proves Lemma~\ref{lemfec2}. In order to do so, we
consider the modification of Theorem~\ref{kahnth}, obtained by:

\smallskip
\qitee{(i)}Adding at the end of the first paragraph of that theorem:\\
``\emph{Suppose furthermore that~$J_1$ is a graph with vertex set $E(H)$
  and maximum degree at most~$\Delta^{1/2}$, $J_2$ is a graph with vertex
  set $E(H)$ and maximum degree at most~$\Delta^{1/4}$, and every list
  $L(e)$ has at most $2\Delta$ elements.}''

\smallskip
\qitee{(ii)}And adding at the end of the last sentence of the theorem:\\
``\emph{so that we can uncolour a set of edges of~$H$ containing at most
  $\frac13\Delta^{9/10}$ edges incident to any vertex $v$ of $H$, to obtain
  a partial edge-colouring of $H$ such that any two coloured edges joined
  by an edge of~$J_1$ receive different colours, and such that any two
  coloured edges joined by an edge of~$J_2$ receive colours that differ by
  at least~$\Delta^{1/4}$.}''

\smallskip\noindent
We call this strengthening Our Theorem. We first show that it implies (the
full version of) Lemma~\ref{lemfec2} and then discuss its proof.

We set $\delta=\half\eps$, let $\beta$ be the corresponding value from
Lemma~\ref{hcd1}, and define $\Delta_{\eps}$ to be $\Delta_{\delta,2\beta}$
(as in Our Theorem). We set $x^{p_\gamma}_e=|L(e)|^{-1}$ for each colour
$\gamma$ and edge $e$ in $H_{\gamma}$, and $x^{p_\gamma}_e=0$ if~$e$ is not
in $H_{\gamma}$. Thus, for each edge~$e$ we have that $\sum_\gamma
x^{p_\gamma}_e=1$. Applying Lemma~\ref{lemkey} together with
Theorem~\ref{mpthm}, we see that each of the edge-vectors $x^{p_\gamma}$ is
in $(1-\delta)\MP(H)$ and hence in $(1-\delta)\MP(H_\gamma)$. Now
Lemmas~\ref{lemhcd7} and~\ref{hcd1} show that there are hard-core
distributions on $H_\gamma$ with marginals $x^{p_\gamma}$ such that, for
every edge~$e$ and colour~$\gamma$, we have
$\lambda^{p_\gamma}(e)\le \beta x^{p_\gamma}_e$. Since, as we saw in
Remark~\ref{rem.lstare2}, for every edge $e$ we have $|L(e)|\ge\half\Delta$
(for $\Delta$ sufficiently large), setting $K=2\beta $, for every edge~$e$
and colour~$\gamma$, we have $\lambda^{p_\gamma}(e)\le\sfrac{K}{\Delta}$.
Hence, conditions (1) and (2) of Our Theorem hold, and applying that result
proves Lemma~\ref{lemfec2}.

\medskip
The key to Kahn's proof of Theorem~\ref{kahnth} above is the following
lemma, \cite[Lemma~3.1]{Kahn00}.

\begin{lemma}[Kahn \cite{Kahn00}]\label{kahn3.1}\mbox{}\\*
  For every $K,\delta>0$, there exist $\xi=\xi_{\delta,K}$ with
  $0<\xi\le\delta$ and~$\Delta_{\delta,K}$ such that the following holds
  for all $\Delta\ge\Delta_{\delta,K}$. Let~$H$ be a multigraph with
  maximum degree at most~$\Delta$, and with a list~$L(e)$ of acceptable
  colours for every edge~$e$. Define the graphs~$H_\gamma$ as before.

  Suppose that for every colour~$\gamma$ we are given a hard-core
  distribution~$p_\gamma$ on the matchings of~$H_\gamma$ with activities
  $\lambda^{p_\gamma}=\lambda_\gamma$ and marginals
  $x^{p_\gamma}=x_\gamma$, satisfying:{

    \smallskip
    \qitee{(1)}for every edge~$e$:
    $\sum\limits_{\gamma\in L(e)}\!x_\gamma(e)>\rme^{-\xi}$, and

    \qitee{(2)}for every colour~$\gamma$ and edge~$e$:
    $\lambda_\gamma(e)\le\sfrac{K}{\Delta}$.}

  \smallskip\noindent
  Then there are matchings~$M_\gamma$ in~$H_\gamma$ for every
  colour~$\gamma$, such that the following holds. If we set
  $H'=H-\bigcup_{\gamma^*}M_{\gamma^*}$ and
  $H'_\gamma=H_\gamma-V(M_\gamma)-\bigcup_{\gamma^*}M_{\gamma^*}$, we form
  a list~$L'(e)$ for every edge $e$ in $H'$ by removing no longer allowed
  colours from~$L(e)$, and we let~$x'_\gamma$ be the marginals
  corresponding to the activities~$\lambda_\gamma$ on~$H'_\gamma$, then we
  have:

  \smallskip
  \qitee{(a)}for every edge~$e$ of~$H'$:
  $\sum\limits_{\gamma\in L'(e)}\!x'_\gamma(e)>\rme^{-\delta}$, and

  \qitee{(b)}the maximum degree of~$H'$ is at most
  $\frac{1+\delta}{1+\xi}\,\rme^{-1}\Delta$.
\end{lemma}

\noindent
Here is a sketch of how we may use Lemma~\ref{kahn3.1} to prove
Theorem~\ref{kahnth}, following Kahn. First fix a suitable number $s$ of
iterations, where we take $s=\lceil\log(8K)\rceil$. Let $\delta_s=1$, and
define $\delta_{s-1}\ge\cdots\ge\delta_1>0$ by setting
$\delta_{i-1}=\xi_{\delta_i,K}$. Also, let $\delta_0=0$ and
$\Delta^*=\Delta_{\delta_1,K}$. We start with a multigraph $H$ with
$\Delta\ge\rme^s\Delta^*$, and with lists of acceptable colours and
distributions satisfying the hypotheses in Theorem~\ref{kahnth}. For
$i=0,1,\ldots,s$ let $\Delta_i=(1+\delta_{i})\rme^{-i}\Delta$ (so
$\Delta_0=\Delta$ and each $\Delta_i\ge\Delta^*$). Set $H^0=H$, and
$H^0_\gamma=H_\gamma$ for all~$\gamma$. Once we have obtained~$H^{i-1}$
and~$H^{i-1}_\gamma$, in iteration~$i$ we do the following.{

  \smallskip
  \qitee{I.} Choose matchings $M^i_\gamma$ in~$H^{i-1}_\gamma$ (for each
  colour $\gamma$) according to the lemma, with $\delta$ as $\delta_{i-1}$
  and $\Delta$ as $\Delta_{i-1}$.

  \smallskip
  \qitee{II.} For each edge $e$ in some matching $M^i_\gamma$, chose
  $\gamma$ independently and uniformly at random from those $\gamma$ for
  which $e\in M^i_\gamma$, and assign colour $\gamma$ to $e$.

  \smallskip
  \qitee{III.} Form~$H^i$ by removing from~$H^{i-1}$ all edges that were
  assigned a colour in step II. For each colour $\gamma$, form~$H^i_\gamma$
  by removing from~$H^{i-1}_\gamma$ all edges that were assigned some
  colour in step II, and all vertices that are incident to any edge that
  was assigned colour~$\gamma$  in step~II.

}\smallskip\noindent
The key point is that the lemma ensures that if its hypotheses hold for
$H^{i-1}$ with $\delta$ as $\delta_{i-1}$ and~$\Delta$ as $\Delta_{i-1}$,
then they hold for $H^i$ with $\delta$ as $\delta_i$ and $\Delta$ as
$\Delta_i$. So we can indeed follow Kahn and iteratively apply the lemma in
this way for $s$ iterations, and colour all but an uncoloured subgraph with
maximum degree at most $\Delta_s=2\rme^{-s}\Delta$, which is at most
$\sfrac{\Delta}{4K}$ by our choice of~$s$.

On the other hand, in the last iteration we still have that for every edge
$e$, the sum of the marginals at $e$ is near 1. Furthermore, we are using
the same activities, so by condition~(2) of the theorem, for each $\gamma$
we have $\lambda^\gamma(e)\le\sfrac{K}{\Delta}$. But since the
distributions are hard-core, $x^{p_\gamma}(e)\le\lambda_\gamma(e)$. (To see
this, observe that
\[x^{p_\gamma}(e)\:=\: \sum_{M:\,e\in M} p(M)\:=\:
\lambda_{\gamma}(e)\!\sum_{M:\,e\in M}\!p(M\sm\{e\})\:\le\:
\lambda_{\gamma}(e),\]
where the sums are over matchings $M$ in $H_{\gamma}$ containing $e$.)
Taken together this implies that~$|L(e)|$ is near $\sfrac{\Delta}{K}$ and
exceeds $\sfrac{\Delta}{2K}$. Hence, we can finish off the colouring
greedily.

This proof is given in \cite[Section~3]{Kahn00}, and is fairly easy to
extract from what is actually written there.

We shall modify this proof to obtain a proof of Our Theorem as follows. To
deal with the conflicts caused by $J_1$ and $J_2$, we choose to uncolour
the conflicting edge which was coloured last, uncolouring both edges if
they were coloured in the same iteration. We need to ensure that the number
of edges of $H$ incident to any given vertex of $H$ which need to be
recoloured due to these conflicts is less than $\frac13\Delta^{9/10}$.

To this end, we shall modify the statement of Lemma~\ref{kahn3.1}, but
first we introduce some notation. In each iteration, for each edge~$e$
of~$H$, we let~$F(e)$ be the set of colours forbidden on~$e$, either
because they were assigned to a neighbour in $J_1$ in a previous iteration,
or because they are too close to a colour assigned to a neighbour in $J_2$
in a previous iteration. For each vertex~$v$ of~$H$, we let~$X_v$ be the
number of edges~$e$ of~$H$ which are assigned a colour $\gamma$ in this
iteration such that: $\gamma \in F(e)$, or $\gamma$ is assigned in this
iteration to a neighbour of~$e$ in~$J_1$, or~$\gamma$ is
within~$\Delta^{1/4}$ of a colour assigned in this iteration to a neighbour
of~$e$ in~$J_2$. For technical reasons, we count in $X_v$ conflicts
involving all colours assigned to edges in this iteration, not just the
colours we finally choose to colour them.

We will use the variant of Lemma~\ref{kahn3.1} in which we add:

\smallskip
\qitee{(i)}At the end of its first paragraph:\\
``\emph{Let $\tilde{\Delta}=8K\Delta$. Suppose further that we have a
  list~$F(e)$ of at most~$3 \tilde{\Delta}^{1/2}$ colours for every
  edge~$e$, and graphs~$J_1$ and~$J_2$ on $E(H)$, where~$J_1$ has degree at
  most~$\tilde{\Delta}^{1/2}$ and~$J_2$ has degree at
  most~$\tilde{\Delta}^{1/4}$, and that every $L(e)$ has at most
  $2\tilde{\Delta}=16K\Delta$ elements.}''

\smallskip
\qitee{(ii)}At the very end an extra new conclusion:\\
``\emph{\,(c) \; for every vertex~$v$, $X_v\le\Delta^{4/5}$.}''

\smallskip\noindent
We call this variant Our Lemma. Since in proving Our Theorem we need only
apply it when the maximum degree bound for $H_i$ is between $\Delta$ and
$\sfrac{\Delta}{8K}$, we see that we will always have the desired upper
bound on the sizes of the lists, by applying the upper bound in Our
Theorem. Also, since we carry out a constant number of iterations, Our
Lemma tells us we need to uncolour only $O(\Delta^{4/5})$ edges of $H$
which are incident with a specific vertex of $H$. So, provided $\Delta$ is
large enough we can use Our Lemma to obtain Our Theorem, just as Kahn used
Lemma \ref{kahn3.1} to prove his main theorem.

\subsubsection{Proving Our Lemma}

It remains to describe how to modify the proof of Lemma \ref{kahn3.1} to
obtain a proof of Our Lemma.

Kahn proves Lemma~\ref{kahn3.1} by applying the Local Lemma to an
independent family of random matchings obtained by, for each $\gamma$,
independently choosing a random matching $M_{\gamma}$ according to the
hard-core distribution $p_\gamma$. By doing so, he shows that he can avoid
a set of bad events.
 
The \emph{bad events} which he avoids by applying the Local Lemma are
defined in the middle of \cite[page~136]{Kahn00}. There are two kinds: an
event~$T_v$ such that its non-occurrence guarantees the degree of a
vertex~$v$ drops sufficiently, and an event~$T_e$ such that its
non-occurrence ensures that the marginals at an edge~$e$ of the hard-core
distribution for the next iteration sum to a number close to~1.

Kahn defines a distance $t>1$ which is a function of~$\delta$ and~$K$ (and
independent of~$\Delta$), and shows that the probability that a bad event
occurs, \emph{given all the edges of every matching $M_{\gamma}$ which are
  at distance at least~$t$ in~$H$ from the vertex or edge indexing the
  event}, is at most~$p$, for some~$p$ which is~$\Delta^{-\omega(1)}$.

He can then apply the Local Lemma, where the set~$\mathcal{S}_{T_z}$ ($z$ a
vertex or an edge) is the set of events indexed by an edge or vertex within
distance~$2t$ of~$z$ (this is done on \cite[pages~136--137]{Kahn00}). The
key point is that this set has size at most $d=2(\Delta+1)\Delta^{2t}$, so
we have $\rme pd=o(1)$.

(A few remarks: Kahn uses~$D$ where we use~$\Delta$, and
$\Delta_1+\Delta_2$ where we use~$t$. The result we have just stated is
\cite[Lemma~6.3]{Kahn00}; the~$\omega(1)$ here is with respect
to~$\Delta$.)

To modify this proof to obtain Our Lemma, we introduce for each vertex~$v$
of~$H$, a new bad event~$T'_v$ that~$X_v$ exceeds~$\Delta^{4/5}$. In each
iteration, along with insisting that all the~$T_e$ and~$T_v$ fail, we also
insist that all the~$T'_v$ fail. In doing so we use the following claim.
For each vertex~$v$ of $H$, let $E^+(v)$ denote the set of edges of $H$
consisting of the edges $e$ incident to $v$ together with the edges
adjacent in $J_1$ or $J_2$ to edges $e$ incident to $v$.

\begin{claim}\label{clai}\mbox{}\\*
  Let $v\in V(H)$. For every colour $\gamma$, let $L_{\gamma}$ be a given
  matching in $H_{\gamma}$, and suppose that the event~$A$ that
  $M_{\gamma}\setminus E^+(v)= L_{\gamma}$ for each $\gamma$ satisfies
  $\pr(A)>0$. Then $\pr(T'_v\mid A)$ is $\Delta^{-\omega(1)}$.
\end{claim}

\noindent
Given the claim, to prove our variant of the lemma, we can use the Local
Lemma, just as Kahn did. However, we have to use a slightly different
dependency graph because the event $T'_v$ depends on the neighbours of $v$
in $J_1$ and $J_2$. Given an event $U$ of the form $T_x$ or $T'_x$ indexed
by a vertex or edge $x$, we let the set $S_U$ consist of all the events
indexed by some $y$ at distance at most $4t$ from $x$ in the graph $H^+$
formed by the union of $H^*$, $J_1$ and $J_2$ (where we identify edges of
$H$ and the vertices of $H^*$ to which they correspond). Note that this
graph has maximum degree at most $2\Delta$.

Just as with the other events, we have a $\Delta^{-\omega(1)}$ bound on the
probability that any event~$T'_v$ holds, given the choice of all the
matching edges at a suitable distance from $v$ in $H^+$ (by applying our
claim to all the choices of~$L_\gamma$ which extend this choice). Also, we
need not worry further about the events $T_v$ and $T_e$. We can therefore
apply the Local Lemma iteratively as in the last section to prove Our
Lemma.

\begin{proof}{\bf\!\!of Claim~\ref{clai}} \ To prove the claim we first
  bound the conditional expected value of~$X_v$. We consider each edge~$e$
  incident to~$v$ separately. We show that the conditional probability
  that~$e$ is in a conflict is $O(\Delta^{-1/2})$. Summing up over all
  edges~$e$ incident to~$v$ yields that the expected value of~$X_v$ is
  $O(\Delta^{1/2})$. We prove this bound for the conflicts involving edges
  coloured in a previous iteration and edges coloured in this iteration
  separately.

  To begin we consider the colours in $F(e)$. We actually show that for any
  edge~$e$, the conditional probability that~$e$ is assigned a colour from
  $F(e)$, given, for each colour~$\gamma$, a matching~$N_\gamma$ not
  containing $e$ such that~$M_\gamma$ is either~$N_\gamma$ or~$N_\gamma+e$,
  is $O(\Delta^{-1/2})$. (We use $N_\gamma+e$ to denote
  $N_\gamma\cup\{e\}$.) Summing up over all the choices for the~$N_\gamma$
  which extend the~$L_\gamma$, then yields the desired result.
  If~$N_\gamma$ contains an edge incident to~$e$, then $N_{\gamma}+e$ is
  not a matching, so $M_\gamma=N_\gamma$. Otherwise, by the definition of a
  hard-core distribution:
  \[\pr(e\in M_{\gamma}\mid M_{\gamma}\in\{N_{\gamma},N_{\gamma}+e\})\:=\:
  \frac{\lambda_\gamma(e)}{1+\lambda_\gamma(e)}\:\le\:
  \lambda_\gamma(e)\:\le\:\frac{K}\Delta.\]
  The conditional probability we want to bound is the sum over all
  colours~$\gamma$ in $F(e)$ of the conditional probability that~$e$ is
  coloured~$\gamma$. For each of these colours, the conditional probability
  that a conflict actually occurs is at most the conditional probability
  that~$e$ is in~$M_\gamma$. Since this is $O(\Delta^{-1})$, and
  $|F(e)|\le3\tilde{\Delta}^{1/2}$, the desired bound follows.

  We next consider conflicts due to both $e$ and a neighbour $f$ in
  $J_1\cup J_2$ being assigned the same colour in this iteration. It is
  enough to show that the conditional probability that~$e$ is assigned the
  same colour as any such uncoloured neighbour $f$ is $O(\Delta^{-1})$. We
  actually show that for any such edge~$f$, the conditional probability
  that~$e$ is assigned the same colour as~$f$, given, for each
  colour~$\gamma$, a matching~$N_\gamma$ not containing $e$ or $f$ such
  that~$M_\gamma$ is in the set
  $\mathcal{N}_{\gamma}^+=
  \{N_\gamma,N_\gamma+e,N_\gamma+f,N_\gamma+e+f\}$, is $O(\Delta^{-1})$.
  Summing up over all the choices for the~$N_\gamma$ which extend
  the~$L_\gamma$, then yields the desired result. We obtain our bound on
  the probability that~$e$ and~$f$ are both assigned the same colour by
  summing the probability they both get a specific colour~$\gamma$ over all
  the at most $2\tilde{\Delta}$ colours in~$L(e)$. For each such colour, as
  in the previous paragraph, we obtain that
  \[\pr(\{e,f\}\subseteq M_{\gamma}\mid M_{\gamma}\in
    \mathcal{N}^+_{\gamma})\:\le\: \lambda_\gamma(e)\lambda_\gamma(f)\:
  \le\:\Bigl(\dfrac{K}\Delta\Bigr)^2.\]
  Summing over our choices for~$\gamma$ yields the desired result.

  If~$f$ is adjacent to~$e$ in~$J_2$, then having picked a colour~$\gamma$
  in $L(e)$ we have at most $2\tilde{\Delta}^{1/4}$ choices for a
  colour~$\gamma'\ne\gamma$ on~$f$ that causes a conflict. Proceeding as
  above with respect to~$\gamma'$ as well as $\gamma$, we can show that the
  conditional probability that $e$ is coloured~$\gamma$ is at
  most~$\sfrac{K}{\Delta}$, and the conditional probability that~$f$ is
  coloured~$\gamma'$, given that $e$ is coloured~$\gamma$, is at
  most~$\sfrac{K}{\Delta}$. Thus the conditional probability that $e$ is
  coloured~$\gamma$ and~$f$ is coloured~$\gamma'$ is at most
  $\bigl(\sfrac{K}{\Delta}\bigr)^2$. Summing over the at most
  $2\tilde{\Delta}$ choices for $\gamma$, the corresponding choices for
  $\gamma'$, and the at most~$\tilde{\Delta}^{1/4}$ choices for $f$, we
  obtain the desired result.

  We next bound the probability that~$X_v$ exceeds $\Delta^{4/5}$, by
  showing that it is concentrated. We note that if we change the choice of
  one~$M_\gamma$, leaving all the other random matchings unchanged, then
  the only new~$J_1$ or~$J_2$ conflicts counted by~$X_v$ involve edges
  coloured with a colour within~$\tilde{\Delta}^{1/4}$ of~$\gamma$. There
  are at most $2\tilde{\Delta}^{1/4}+1$ such edges incident to~$v$. Thus,
  such a change can change~$X_v$ by at most $2\tilde{\Delta}^{1/4}+1$.
  Furthermore, each conflict involves at most two of the matchings (only
  one if it also involves a previously coloured vertex). So, to certify
  that there were at least~$x$ conflicts involving edges incident to~$v$ in
  an iteration we need only produce at most~$2x$ matchings involved in
  these conflicts. It follows by a result of Talagrand~\cite{Tal95} (see
  also \cite[Chapter~10]{MoRe02}) that the probability that~$X_v$ exceeds
  its median~$M$ by more than~$t$ is at most
  \[\exp\Bigl(-\Omega\dfrac{t^2}{\Delta^{1/2}M}\Bigr).\]
  Since the median of~$X_v$ is at most twice its expectation, setting
  $t=\half\Delta^{4/5}$ yields the desired result.

  This completes the proof of the claim, and hence of Our Lemma.
\end{proof}

\noindent
Now that Our Lemma has been proved, we can deduce Our Theorem,
Lemma~\ref{lemfec2} and Lemma~\ref{lemfec}. All that remains is to prove
Lemma~\ref{lemnotover2}, which we will do now.


\subsection{The Final Stage: Deriving Lemma \ref{lemnotover2}}

With Lemma~\ref{lemfec} in hand, it is an easy matter to prove
Lemma~\ref{lemnotover2}. In doing so we consider the natural bijection
between the core~$R$ of~$H^*$ and~$E(H)$, referring to these objects using
whichever terminology is convenient. (We sometimes use both names for the
same object in the same sentence.) Similarly, we use the same letter to
denote a vertex of $H$ and the corresponding vertex of $G$.

Before we really start, we make one observation concerning degrees. For a
vertex~$v$ in~$H$, the condition in Lemma~\ref{lemnotover2}, taking
$X=\{v\}$, gives $d_{G-R}(v)-d_H(v)\le\frac1{30}\eps\Delta$. Since
$d_{G-R}(v)=d_G(v)-d_H(v)$, this means that
$d_H(v)\ge\half d_G(v)-\frac1{60}\eps\Delta$, and hence
\[d_G(v)-d_H(v)\:\le\: \half d_G(v)+\tfrac1{60}\eps\Delta\:\le\:
\bigl(\half+\tfrac1{60}\eps\bigr)\Delta.\]
This bound will guarantee that all the lists of colours we will consider
below are not empty.

Starting with $J_1$ and $J_2$ empty, for every two vertices~$x,y$ from~$R$,
if~$x$ and~$y$ are adjacent in~$G$, we add the edge~$xy$ to~$J_2$, and
if~$x$ and~$y$ are adjacent in~$G^2$, but do not correspond to incident
edges in~$H$, then we add the edge~$xy$ to~$J_1$. Since vertices in~$R$
have degree at most~$\Delta^{1/4}$ in~$G$, we get the required bounds on
the degree for vertices in~$J_1$ and~$J_2$ in Lemma~\ref{lemfec}.

Now first suppose that every vertex~$v$ in~$H$ has degree~$\Delta$ in~$G$.
Recall the definition of~$\lstar_e$ in equation~\eqref{eqn.lstaredef}. For
an edge $e=vw$ in~$H$, set~$L'(e)$ to be a subset of $\lstare$ colours
in~$L(e)$ which appear on no vertex of $V\sm R$ which is a neighbour
of~$e^*$ in~$G^2$ and are not within~$\Delta^{1/4}$ of any colour appearing
on a neighbour of~$e^*$ in~$G$. This is possible because $e^*$ is adjacent
in~$G^2$ to at most $(\Delta-d_H(v))+(\Delta-d_H(w))$ neighbours of~$v^*$
and~$w^*$ in $V-R$, and at most~$\Delta^{1/2}$ other vertices of $V-R$
(since the vertex $e^*$ in~$G$ is removable, hence has at
most~$\Delta^{1/4}$ neighbours other than~$v^*$ and $w^*$, and all these
vertices have degree at most~$\Delta^{1/4}$). Finally, condition~(2) in
Lemma~\ref{lemfec} holds because of the corresponding condition for all
sets~$X$ in the statement of Lemma~\ref{lemnotover2}. So applying
Lemma~\ref{lemfec}, we are done in this case.

In general this approach does not work because for a vertex~$v$ of~$H$ with
degree less than~$\Delta$, we do not have that $\Delta-d_H(v)$ is equal to
the number of edges from~$v$ to $V-R$, so our two conditions are not quite
equivalent. In order to fix this, we use a simple trick. Form~$\ov{G}$ by
taking two disjoint copies~$\cpa{G}$ and~$\cpb{G}$ of~$G$, with
corresponding copies $\cpi{H}$, $\cpi{R}$, $\cpi{J_1}$,~$\cpi{J_2}$,
$i=1,2$, and copy the lists of colours on the vertices of $G$ to the two
copies of these vertices. For each vertex~$v$ of~$H$, we add
$\Delta-d_G(v)$ subdivided edges between its two copies~$\cpa{v}$
and~$\cpb{v}$. Give an arbitrary list of
$\bigl\lceil\bigl(\frac32+\eps\bigr)\Delta\bigr\rceil$ colours to the
vertices at the middle of these new subdivided edges.

Let~$\ov{H}$ be the multigraph formed by the union of~$\cpa{H}$
and~$\cpb{H}$ together with multiple edges corresponding to the new
subdivided edges between copies of vertices of~$H$. Similarly,
take~$\ov{R}$ the union of $\cpa{R}$,~$\cpb{R}$ and all vertices in the
middle of the new subdivided edges, and set
$\ov{J_i}=\cpa{J_i}\cup\cpb{J_i}$ for $i=1,2$. Note that the degrees
in~$\ov{J_1}$ and~$\ov{J_2}$ haven't changed, so we can still use them in
Lemma~\ref{lemfec}.

Recall that for $i\in\{1,2\}$ and all $v\in\cpi{H}$, we have
$\Delta-d_{\ov{H}}(v)=d_G(v)-d_{\cpi{H}}(v)$. Now we choose lists of
colours on the edges of~$\ov{H}$. Each new edge $\cpa{v}\cpb{v}$ gets an
arbitrary list of
$\bigl\lceil\bigl(\frac32+\eps\bigr)\Delta-
(\Delta-d_{\ov{H}}(\cpa{v}))-(\Delta-d_{\ov{H}}(\cpb{v}))-
3\Delta^{1/2}\bigr\rceil=
\bigl\lceil\bigl(\frac32+\eps\bigr)\Delta-2(d_G(v)-d_H(v))-
3\Delta^{1/2}\bigr\rceil$ colours from the
$\bigl\lceil\bigl(\frac32+\eps\bigr)\Delta\bigr\rceil$ colours we gave on
the vertex in the middle of it. On the two copies of an edge $e=vw$ of~$H$
we take the same list of
$\bigl\lceil\bigl(\frac32+\eps\bigr)\Delta-(\Delta-d_{\ov{H}}(v))-
(\Delta-d_{\ov{H}}(w))-3\Delta^{1/2}\bigr\rceil$ colours. Since this is
equal to
$\bigl\lceil\bigl(\frac32+\eps\bigr)\Delta-(d_G(v)-d_{H}(v))-
(d_G(w)-d_H(w))-3\Delta^{1/2}\bigr\rceil$, we can still choose this list to
be disjoint from the colours used on the neighbours of this edge in
$G^2\!-\!R$.

We note that if we can find a proper colouring of~$L(\ov{H})$ using the
chosen lists which avoids conflicts, then we get two (possibly identical)
extensions of our colouring of~$G^2-R$ to~$G^2$. We apply
Lemma~\ref{lemfec} to prove that we can indeed find such an acceptable
colouring. To do so, we only need to show that for every odd set~$X$ of
vertices of~$\ov{H}$, we have
\[\sum_{v\in X}(\Delta-d_{\ov{H}}(v))-e(X,V(\ov{H})\sm X)\:\le\:
\tfrac1{30}\eps|X|\Delta.\]
In fact, we will do this for all subsets~$X$ of $V(\ov{H})$. We set
$\cpi{X}=X\cap V(\cpi{H})$, $i=1,2$. We immediately get that
$e(X,V(\ov{H})\sm X)\:\ge\:e(\cpa{X},V(\cpa{H})\sm\cpa{X})+
e(\cpb{X},V(\cpb{H})\sm\cpb{X})$ (since on the right hand right we are
ignoring the edges between the two copies of~$H$). Recall that
$\Delta-d_{\ov{H}}(v)=d_{G-H}(v)$ for a vertex~$v$ in~$\ov{H}$. Using the
condition in Lemma~\ref{lemnotover2} for the two copies of~$H$, this gives
\begin{align*}
  \sum_{v\in X}(\Delta-d_{\ov{H}}(v)&)-e(X,V(\ov{H})\sm X)\\
  {}\le\:{}&\sum_{v\in\cpa{X}}d_{G-H}(v)+\sum_{v\in\cpb{X}}d_{G-H}(v)\\
  &\qquad{}-e(\cpa{X},V(\cpa{H})\sm\cpa{X})-
  e(\cpb{X},V(\cpb{H})\sm\cpb{X})\\
  {}\le\:{}&\tfrac1{30}\eps|\cpa{X}|\Delta+
  \tfrac1{30}\eps|\cpb{X}|\Delta\:=\: \tfrac1{30}\eps|X|\Delta,
\end{align*}
and we are done.\proofend


\section{Conclusions and Discussion}

In this paper, we showed that the chromatic number $\chi(G^2)$ of the
square of a graph~$G$ from a fixed nice family is at most
$\bigl(\frac32+o(1)\bigr)\Delta(G)$. But many questions remain.

One can prove a bound of constant times the maximum degree for the
chromatic number of the square of a graph from a minor-closed family.
Krumke \emph{et al}.~\cite{KMR98} showed that if a graph~$G$ is
\textit{$q$-degenerate} (there exists an ordering $v_1,v_2,\dots,v_n$ of
the vertices such that every~$v_i$ has at most~$q$ neighbours in
$\{v_1,\dots,v_{i-1}\}$), then its square is $((2q-1)\Delta(G))$-degenerate
--- the same ordering does the job. But for every minor-closed
family~$\MF$, there is a constant~$C_\MF$ such that every graph in~$\MF$ is
$C_\MF$-degenerate (see Theorem~\ref{adbounded} and the first paragraph of
Section~\ref{seceuler}). Hence~$G^2$ is $((2C_\MF-1)\Delta(G))$-degenerate
for every $G\in\MF$ and so its list chromatic number is at most
$(2C_\MF-1)\Delta(G)+1$.

But it is unlikely that this is the best possible bound.

\begin{question}\mbox{}\\*
  For a given minor-closed family~$\MF$ graphs (not the set of all graphs),
  what is the smallest constant~$D_\MF$ so that
  $\chi(G^2)\le(D_\MF+o(1))\Delta(G)$ for all $G\in\MF$?
\end{question}

\noindent
The following examples show that for~$\MF$ the class of
$K_{4,4}$-minor-free graphs we must have $D_\MF\ge2$. Let $V_1,\ldots,V_4$
be four disjoint sets of~$m$ vertices, and let
$X=\{x_{12},x_{13},x_{14},x_{23},x_{24},x_{34}\}$ be a further six
vertices. Let~$G_m$ be the graph with vertex set
$X\cup V_1\cup\cdots\cup V_4$, and edges between any~$x_{ij}$ and all
vertices in $V_i\cup V_j$, $1\le i<j\le4$. It is easy to check that~$G_m$
is $K_{4,4}$-minor-free. For $m\ge2$ we have
$\Delta(G_m)=d_{G_m}(x_{ij})=2m$. Moreover, all vertices in
$V_1\cup\cdots\cup V_4$ are adjacent in~$G_m^2$, and hence
$\chi(G_m^2)\ge4m=2\Delta(G_m)$. (Of course, $G_m$ has $K_{3,m}$ as a
minor, so we do not have a contradiction to Theorem~\ref{mt}.)

It is easy to generalise these examples to show that for~$\MF$ the class of
$K_{k,k}$-minor-free graphs, $k\ge3$, we must have $D_\MF\ge\half k$.

But even for nice classes of graphs, many open problems remain. Our proof
of the upper bound on the (list) chromatic number does not provide an
efficient algorithm. So, for a nice family~$\MF$, it would be interesting
to find an efficient algorithm to find a colouring of the square of a graph
$G\in\MF$ with at most $\bigl(\frac32+o(1)\bigr)\Delta(G)$ colours.

Moreover, our result suggests that Wegner's Conjecture(see
Conjectures~\ref{conj.W1} and~\ref{conj.W2}) should be generalised to nice
families of graphs and to list colouring.

\begin{conjecture}\label{general}\mbox{}\\*
  Let~$\MF$ be a nice family of graphs. Then for any graph $G\in\MF$
  with~$\Delta(G)$ sufficiently large,
  $\chi(G^2)\le\ch(G^2)\le
  \bigl\lfloor\frac32\Delta(G)\bigr\rfloor+1$
\end{conjecture}

\noindent
The results of Lih, Wang and Zhu~\cite{LWZ03} and Hetherington and
Woodall~\cite{HW07} show that the conjecture is true when~$\MF$ is the
family of $K_4$-minor-free graphs.

As $\omega(G^2)\le\chi(G^2)$, our result implies
$\omega(G^2)\le\bigl(\frac32+o(1)\bigr)\Delta(G)$ for $G$ in a nice family.
But does there exist a simple proof showing this inequality?

Hell and Seyffart~\cite{HS93} proved that for $\Delta\ge8$, a planar graph
with maximum degree~$\Delta$ and diameter two has at most
$\bigl\lfloor\frac32\Delta\bigr\rfloor+1$ vertices. Using their proof
techniques it is not so hard to show that, for sufficiently large~$\Delta$,
a planar graph~$G$ with maximum degree~$\Delta$ satisfies
$\omega(G^2)\le\bigl\lfloor\frac32\Delta\bigr\rfloor+1$. (In fact, with
some effort we expect this bound can be proved for all $\Delta\ge8$.) Note
that this last inequality is tight as shown by the examples of
Figure~\ref{tightfig}.

Corollary~1.7 in~\cite{AEH2013} says that, for each fixed surface $S$,
there is a constant $c_S$ such that $\omega(G^2)\le\frac32\Delta(G)+c_S$
for each graph $G$ embeddable in $S$. More generally, can we prove that,
for each nice family~$\MF$, there is a constant $c_{\MF}$ such that
$\omega(G^2)\le\frac32\Delta(G)+c_{\MF}$ for each graph $G$ in~$\MF$? Can
we take $c_{\MF}=1$ for sufficiently large $\Delta$ (depending on $\MF$)?

A major part of the proof of our result is a reduction to list
edge-colouring of line graphs. For edge-colourings, Kahn~\cite{Kahn00}
proved that asymptotically the list chromatic number equals the fractional
chromatic number. This may suggest that the same could be true for squares
of planar graphs, or more generally for squares of graphs of a nice family.

\begin{problem}\mbox{}\\*
  Given a nice family $\MF$ of graphs, is it true that
  $\ch(G^2)=(1+o(1))\chi_f(G^2)$ for $G\in\MF$?
\end{problem}

\noindent
As already mentioned in the introduction, we believe that for every square
of a planar graph the list chromatic number equals the chromatic number.

\begin{conjecture}\mbox{}\\*
  If $G$ is a planar graph, then $\ch(G^2)=\chi(G^2)$.
\end{conjecture}

\noindent
Since there are graphs $G$ for which $\ch(G^2)>\chi(G^2)$, a natural
problem is to determine the best possible upper bound on $\ch(G^2)$ in
terms of $\chi(G^2)$ for general graphs. Since
$\Delta(G)+1\le\chi(G^2)\le\ch(G^2)\le\Delta(G)^2+1$, we trivially have
$\ch(G^2)\le(\chi(G^2))^2$. However, this trivial quadratic upper bound is
certainly not best possible. Kosar \emph{et al}.~\cite{KSRY14} posed the
following question.

\begin{problem}[Kosar, Petrickova, Reiniger, and
  Yeager~\cite{KSRY14}]\mbox{}\\*
 Is there a function $f(k)=o(k^2)$ such that for every graph $G$,
 $\ch(G^2)\le f(\chi(G^2))$?
\end{problem}

\noindent
They also formulated the following more specific question.

\begin{problem}[Kosar, Petrickova, Reiniger, and
  Yeager~\cite{KSRY14}]\label{pr:kos}\mbox{}\\*
  Does there exist a constant $C$ such that every graph $G$ satisfies
  $\ch(G^2)\le C\cdot\chi(G^2)\log(\chi(G^2))$? 
\end{problem}

\noindent
If the answer to Problem~\ref{pr:kos} is ``yes'', then the upper bound will
be tight, up to the value of the constant~$C$, as Kosar \emph{et
  al}.~\cite{KSRY14} constructed an infinite family of graphs $G$ with
unbounded $\chi(G^2)$ such that
$\ch(G^2)\ge C'\cdot\chi(G^2)\log(\chi(G^2))$ for some constant $C'$.

\medskip
Finally, our proof uses Kahn's proof of his theorem that the list chromatic
index $\ch'(G)$ of a graph $G$ is $(1+o(1))\chi'_f(G)$, and that theorem of
course implies that $\ch'(G)=(1+o(1))\chi'(G)$. This is an asymptotic
version of the celebrated List Colouring Conjecture.

\begin{conjecture}[List Colouring Conjecture]\mbox{}\\*
  For every graph $G$, $\ch'(G)=\chi'(G)$.
\end{conjecture}

\noindent
A more general conjecture was made by Gravier and Maffray{\cite{GrMa97}},
who conjectured that for every claw-free graph, the list chromatic number
equals the chromatic number. It is possible that advances on the List
Colouring Conjecture might be helpful towards Wegner's Conjecture.



\begin{thebibliography}{XX}

\bibitem{AgHa00} G. Agnarsson and M.M. Halld\'orsson, \textsl{Coloring
    powers of planar graphs}. SIAM~J.\ Discrete\ Math.~\textbf{16} (2003),
  651--662.

\bibitem{AEH2013}
O.~Amini, L.~Esperet and J.~van den Heuvel,
\textsl{A unified approach to distance-two colouring of graphs on surfaces}.
Combinatorica \textbf{33} (3) (2013), 253 -- 296.

\bibitem{ApHa77a} K. Appel and W. Haken, \textsl{Every planar map is four
    colorable. I.~Discharging}. Illinois~J. Math.~\textbf{21} (1977),
  429--490.

\bibitem{ApHa77b} K. Appel, W. Haken, and J. Koch, \textsl{Every planar map
    is four colorable. II.~Reducibility}. Illinois~J. Math.~\textbf{21}
  (1977), \ 491--567.

\bibitem{ApHa89} K. Appel and W. Haken, \textsl{Every Planar Map is Four
    Colorable}. Contemp. Math.~\textbf{98}. American Mathematical Society,
  Providence, RI, 1989.

\bibitem{BoMu76} J.A. Bondy and U.S.R. Murty, \textsl{Graph Theory}, 2nd
  edition. Springer-Verlag, Berlin and Heidelberg, 2008.

\bibitem{Bor+} O.V. Borodin, H.J. Broersma, A. Glebov, and J. van den
  Heuvel, \textsl{Minimal degrees and chromatic numbers of squares of
    planar graphs} (in Russian). Diskretn.\ Anal.\ Issled.\ Oper.\
  Ser.~1~\textbf{8}, no.~4 (2001), 9--33.

\bibitem{CrKi} D.W. Cranston and S.-J. Kim, \textsl{List-coloring the
    square of a subcubic graph}. J.~Graph Theory~\textbf{57} (2008),
  65--87.

\bibitem{Die05} R. Diestel, \textsl{Graph Theory}, 4th edition.
  Springer-Verlag, Berlin and Heidelberg, 2010.

\bibitem{Edm65} J. Edmonds, \textsl{Maximum matching and a polyhedron with
    $0,1$-vertices}. J.~Res.\ Nat.\ Bur.\ Standards Sect.~B~\textbf{69B}
  (1965), 125--130.

\bibitem{ErLo75} P. Erd\H{o}s and L. Lov\'asz, \textsl{Problems and results
    on 3-chromatic hypergraphs and some related questions}. In:
  \textsl{Infinite and Finite Sets}, 609--627. Colloq.\ Math, Soc.\
  J.~Bolyai~\textbf{10}. North-Holland, Amsterdam (1975).

\bibitem{GrMa97} S. Gravier and F. Maffray. \textsl{Choice number of
    3-colorable elementary graphs}. Discrete Math.~\textbf{165/166}(1997),
  353--358.

\bibitem{HJT16} S.G. Hartke, S. Jahanbekam, and B. Thomas, \textsl{The
    chromatic number of the square of subcubic planar graphs}.
  arXiv:1604.06504 [math.CO] (2016).

\bibitem{HS93} P. Hell and K. Seyffart, \textsl{Largest planar graphs of
    diameter two and fixed maximum degree}. Discrete Math.~\textbf{111}
  (1993), 312--322.

\bibitem{vdHMG} J. van den Heuvel and S. McGuinness, \textsl{Coloring the
    square of a planar graph}. J.~Graph Theory~\textbf{42} (2003),
  110--124.

\bibitem{HW07} T.J. Hetherington and D.R. Woodall, \textsl{List-colouring
    the square of a $K_4$-minor-free graph}. Discrete Math.~\textbf{308}
  (2008), 4037--4043.

\bibitem{JeTo95} T.R. Jensen, B. Toft, \textsl{Graph Coloring Problems}.
  John Wiley~\& Sons, New York, 1995.

\bibitem{Jon93} T.K. Jonas, \textsl{Graph Coloring Analogues with a
    Condition at Distance Two: $L(2,1)$-Labelings and List
    $\lambda$-Labelings}. Ph.D.~Thesis, University of South Carolina, 1993.

\bibitem{Kahn00} J. Kahn, \textsl{Asymptotics of the list-chromatic index
    for multigraphs}. Random Structures Algorithms~\textbf{17} (2000),
  117--156.

\bibitem{KaKa97} J. Kahn and P.M. Kayll, \textsl{On the stochastic
    independence properties of hard-core distributions}.
  Combinatorica~\textbf{17} (1997), 369--391.

\bibitem{Kay01} P.M. Kayll, \textsl{Asymptotically good choice numbers of
    multigraphs}. Ars Combin.~\textbf{60} (2001), 209--217.

\bibitem{KKP15} S.-J. Kim, Y.S. Kwon, and B. Park,
  \textsl{Chromatic-choosability of the power of graphs}. Discrete Appl.
  Math.~\textbf{180} (2015), 120--125.

\bibitem{KP15} S.-J. Kim and B. Park, \textsl{Counterexamples to the list
    square coloring conjecture}. J.~Graph Theory~\textbf{78} (2015),
  239--247.
 
\bibitem{KP15b} S.-J. Kim and B. Park, \textsl{Bipartite graphs whose
    squares are not chromatic-choosable}. Electron. J. Combina.~\textbf{22}
  (2015), \# P1.46.

\bibitem{KSTM05} A. Kohl, J. Schreyer, Z. Tuza, and M. Voigt, \textsl{List
    version of $L(d,s)$-labelings}. Theoret.\ Comput.\ Sci.~\textbf{349}
  (2005), 92--98.

\bibitem{KSRY14} N. Kosar, S. Petrickova, B. Reiniger, and E. Yeager,
  \textsl{A note on list-coloring powers of graphs}. Discrete
  Math.,~\textbf{332} (2014), 10--14.

\bibitem{Kos84} A.V. Kostochka, \textsl{Lower bounds of the Hadwiger number
    of graphs by their average degree}. Combinatorica~\textbf{4} (1984),
  307--316.

\bibitem{KoWo01} A.V. Kostochka and D.R. Woodall, \textsl{Choosability
    conjectures and multicircuits}. Discrete Math.~\textbf{240} (2001),
  123--143.

\bibitem{Kral} D. Kr\'al', \textsl{Channel assignment problem with variable
    weights}. SIAM~J.\ Discrete\ Math.~\textbf{20} (2006), 690--704.

\bibitem{KMR98} S.O. Krumke, M.V. Marathe, and S.S. Ravi,
  \textsl{Approximation algorithms for channel assignment in radio
    networks}. In: \textsl{Proceedings of the 2nd International Workshop on
    Discrete Algorithms and Methods for Mobile Computing and Communications
    (Dial M for Mobility). Dallas, Texas}, 1998.

\bibitem{Lee90} C. Lee, \textsl{Some recent results on convex polytopes}.
  In: \textsl{Mathematical developments arising from linear programming},
  3--19. Contemp. Math.~\textbf{114}. American Mathematical Society,
  Providence, RI, 1990.
    
\bibitem{LeeseHurley} R.~Leese and S.~Hurley, \textsl{Methods and Models
    for Radio Channel Assignment}. Oxford Lecture Ser. Math.
  Appl.~\textbf{23}. Oxford Univ. Press, Oxford, 2002.

\bibitem{LWZ03} K.-W. Lih, W.F. Wang, and X. Zhu, \textsl{Coloring the
    square of a $K_4$-minor free graph}. Discrete Math.~\textbf{269}
  (2003), 303--309.

\bibitem{LiZh05} D.D.-F. Liu and X. Zhu, \textsl{Multilevel distance
    labelings for paths and cycles}. SIAM~J.\ Discrete\ Math.~\textbf{19}
  (2005), 610--621.

\bibitem{Mad68} W. Mader, \textsl{Homomorphies\"atze f\"ur Graphen}, Math.\
  Ann.~\textbf{178} (1968), 154--168.

\bibitem{McD03} C. McDiarmid, \textsl{On the span in channel assigment
    problems: bounds, computing and counting}. Discrete
  Math.~\textbf{266} (2003), 387--397.

\bibitem{MoRe02} M. Molloy and B. Reed, \textsl{Graph Colouring and the
    Probabilistic Method}. Algorithms and Combinatorics~\textbf{23}.
  Springer-Verlag, Berlin, 2002.

\bibitem{MoSa02} M. Molloy and M.R. Salavatipour, \textsl{A bound on the
    chromatic number of the square of a planar graph}. J.~Combin.\ Theory
  Ser.~B~\textbf{94} (2005), 189--213.

\bibitem{NSTW} S.~Norine, P.~Seymour, R.~Thomas, and P.~Wollan,
  \textsl{Proper minor-closed families are small}. J.~Combin.\ Theory
  Set.~B~\textbf{96} (2006), 754--757.

\bibitem{RSW92} Y. Rabinovitch, A. Sinclair, and A. Widgerson,
  \textsl{Quadratic dynamical systems}. In: \textsl{Proceedings of the
    Thirty-Third Annual IEEE Symposium on Foundations of Computer Science
    (FOCS)} (1992), 304--313.

\bibitem{Tal95} M. Talagrand, \textsl{Concentration of measure and
    isoperimetric inequalities in product spaces}. Inst.\ Hautes \'Etudes
  Sci. Publ.\ Math.~\textbf{81} (1995), 73--205.

\bibitem{Tho84} A.G. Thomason, \textsl{An extremal function for
    contractions of graphs}. Math.\ Proc.\ Camb.\ Phil.\ Soc.~\textbf{95}
  (1984), 261--265.

\bibitem{Tho06} C. Thomassen, \textsl{The square of a planar cubic graph is
    7-colorable}. Manuscript.

\bibitem{Weg77} G. Wegner, \textsl{Graphs with given diameter and a
    coloring problem}. Technical Report, University of Dortmund, 1977.

\bibitem{Won96} S.A. Wong, \textsl{Colouring Graphs with Respect to
    Distance}. M.Sc.~Thesis, Department of Combinatorics and Optimization,
  University of Waterloo, 1996.

\bibitem{Wo01} D.R. Woodall, \textsl{List colourings of graphs}. In:
  \textsl{Surveys in Combinatorics, 2001}, 269--301. London Math.\ Soc.\
  Lecture Note Ser.~\textbf{288}, Cambridge Univ.\ Press, Cambridge, 2001.

\bibitem{Yeh06} R.K. Yeh, \textsl{A survey on labeling graphs with a
    condition at distance two}. Discrete Math.~\textbf{306} (2006),
  1217--1231.

\end{thebibliography}
\end{document}